\documentclass[12pt]{article}
\usepackage[english]{babel}
\usepackage{amsmath,amsthm,amsfonts,amssymb,epsfig,verbatim}
\usepackage[left=1in,top=1in,right=1in]{geometry}
\usepackage[normalem]{ulem}
\usepackage{color}

\newtheorem{thm}{Theorem}[section]
\newtheorem{lem}[thm]{Lemma}
\newtheorem{prop}[thm]{Proposition}

\newtheorem{df}[thm]{Definition}

\newtheorem*{rem}{Remark}

\newtheorem{question}{Question}

\numberwithin{equation}{section}

\newcommand{\E}{\mathbf{E}}

\newcommand{\prob}{\mathbf{P}}

\newcommand{\R}{\mathbb{R}}

\newcommand{\Z}{\mathbb{Z}}
\newcommand{\N}{\mathbb{N}}

\DeclareMathOperator{\Ann}{Ann}
\DeclareMathOperator{\Leb}{Leb}
\DeclareMathOperator{\Vol}{Vol}

\makeatletter
\def\iddots{\mathinner{\mkern1mu\raise\p@
		\vbox{\kern7\p@\hbox{.}}\mkern2mu
		\raise4\p@\hbox{.}\mkern2mu\raise7\p@\hbox{.}\mkern1mu}}
\makeatother

\title{The size of the boundary in first-passage percolation}
\author{Michael Damron \thanks{The research of M. D. is supported by NSF grant DMS-0901534 and an NSF CAREER grant.} \\ \small{Georgia Tech}  \and Jack Hanson \thanks{The research of J. H. is supported by NSF grant DMS-1612921.}\\ \small{City College of New York} \and Wai-Kit Lam \\ \small{Indiana University Bloomington}}

\begin{document}

\maketitle 
\begin{abstract}
First-passage percolation is a random growth model defined using i.i.d. edge-weights $(t_e)$ on the nearest-neighbor edges of $\mathbb{Z}^d$. An initial infection occupies the origin and spreads along the edges, taking time $t_e$ to cross the edge $e$. In this paper, we study the size of the boundary of the infected (``wet'') region at time $t$, $B(t)$. It is known that $B(t)$ grows linearly, so its boundary $\partial B(t)$ has size between $ct^{d-1}$ and $Ct^d$. Under a weak moment condition on the weights, we show that for most times, $\partial B(t)$ has size of order $t^{d-1}$ (smooth). On the other hand, for heavy-tailed distributions, $B(t)$ contains many small holes, and consequently we show that $\partial B(t)$ has size of order $t^{d-1+\alpha}$ for some $\alpha>0$ depending on the distribution. In all cases, we show that the exterior boundary of $B(t)$ (edges touching the unbounded component of the complement of $B(t)$) is smooth for most times. Under the unproven assumption of uniformly positive curvature on the limit shape for $B(t)$, we show the inequality $\#\partial B(t) \leq (\log t)^C t^{d-1}$ for all large $t.$
\end{abstract}

	\section{Introduction}	
	
	In this paper, we study properties of the boundary of the growing set in first-passage percolation (FPP), a random growth model. Consider the graph $(\Z^d, \mathcal{E}^d)$ for $d \geq 2$, where $\mathcal{E}^d$ is the set of nearest-neighbor edges of $\mathbb{Z}^d$. FPP is defined as follows. Let $(t_e)_{e\in\mathcal{E}^d}$ be a family of i.i.d. nonnegative random variables. We define a finite path as an alternating sequence of vertices and edges $(x_0,e_1,x_1,\ldots, e_n, x_n)$, where $x_i\in \Z^d$ and $e_i=\{x_{i-1}, x_i\}\in\mathcal{E}^d$, and an infinite path as an infinite alternating sequence $(x_0,e_1,x_1,\ldots)$. For $x,y\in\Z^d$, define the first-passage time from $x$ to $y$ by
	\[
	T(x,y) = \inf_{\gamma  : x \to y}T(\gamma),
	\]
	where the infimum is over all lattice paths $\gamma$ from $x$ to $y$, and $T(\gamma):=\sum_{e\in\gamma} t_e$. Then $T(\cdot,\cdot)$ defines a pseudometric on $\Z^d$. Consider
	\[
	B(t) = \{x\in\Z^d: T(0,x)\leq t\},
	\]
	the ball centered at the origin with radius $t \geq 0$. Of interest are the geometric properties of $B(t)$ when $t$ is large. Motivated by a question of K. Burdzy, which appeared later in \cite{Burdzy} (see some earlier references listed below), we aim to describe the size of the boundary of $B(t)$, and to determine if it is surface-like (smooth) or fractal-like (rough). We refer the reader to the survey \cite{survey} for other aspects of FPP.
	
	We will consider two types of boundaries, the edge boundary and the edge exterior boundary. 
	
	\begin{df}
		Let $V\subseteq \Z^d$.
		\begin{enumerate}
			\item The edge boundary of $V$ is the set
			\[
			\partial_{\textnormal{e}} V= \{\{x,y\} \in \mathcal{E}^d: x\in V,\;y\in\Z^d\setminus V\}.
			\]
			\item The vertex exterior boundary $\partial^{\textnormal{ext}} V$ of $V \subseteq \Z^d$ is the set of all $x \in \Z^d \setminus V$ which are
			\begin{enumerate}
				\item adjacent to a vertex in $V$, and
				\item the starting point of some infinite vertex self-avoiding path which does not intersect $V$.
			\end{enumerate}
			
			The edge exterior boundary $\partial^{\textnormal{ext}}_{\textnormal{e}} V$ of a set $V\subseteq \mathbb{Z}^d$ is the set of edges $\{x, y\}$ for some $y \in V$ and $x \in \partial^{\textnormal{ext}} V$.
		\end{enumerate}		
	\end{df}
	
	Write $\#V$ for the cardinality of a set $V$. The specific question we address is:
	\begin{center}
		{\it What is the typical order of $\#\partial_{\textnormal{e}} B(t)$ or $\#\partial^{\textnormal{ext}}_{\textnormal{e}} B(t)$?}
	\end{center}

	We can obtain some straightforward bounds from shape theorems, which were first proved by Richardson \cite{R} and Cox-Durrett \cite{CD} with weaker forms extended to higher dimensions by Kesten \cite{kesten2}. To state a shape theorem, we first extend $T$ to $\R^d\times\R^d$ by defining $T(x,y) = T([x],[y])$ for $x,y\in \R^d$, where $[x]$ is the unique vertex in $\Z^d$ such that $x\in [x]+ [0,1)^d$ (similarly for $[y]$).
	Let 
	\[
	\bar{B}(t) = \{x\in \R^d: T(0,x)\leq t\}
	\]
	and let $p_c=p_c(d)$ be the critical threshold for Bernoulli bond percolation on $\Z^d$ (see \cite{Grimmett}). If $\prob(t_e=0)<p_c$, then there exists a nonrandom, compact, convex set $\mathcal{B}\subseteq\R^d$ with nonempty interior and with the symmetries of $\mathbb{Z}^d$ that fix the origin, such that almost surely,
	\begin{equation}
	\label{eq:shape_thm}
	\Vol\left(\frac{\bar{B}(t)}{t}\Delta \mathcal{B}\right)\to 0 \quad \text{as $t\to\infty$.}
	\end{equation}
	Here $\Delta$ is the symmetric difference, $\Vol$ is the $d$-dimensional volume, and we use the notation $cA = \{ca: a \in A\}$ for $A \subseteq \mathbb{R}^d$ and $c \in \mathbb{R}$. Using the fact that $\Vol(\bar{B}(t)) = \# B(t)$, we can easily obtain from \eqref{eq:shape_thm} that there exist $c_1, c_2>0$ such that almost surely, $c_1t^d\leq \# B(t) \leq c_2t^d$ for all large $t$. Together with the isoperimetric inequality and the fact that $\#\partial_{\textnormal{e}} V \leq 2d\# V$,  we can show that there exists a constant $c_3>0$ such that almost surely,
	\begin{equation}
	\label{eq:volume}
	c_3 t^{d-1} \leq \#\partial_{\textnormal{e}} B(t) \leq 2dc_2 t^d \;\text{ for all large $t$.}
	\end{equation}
	(Similar inequalities hold for the exterior boundary.) In fact, one can even deduce from \eqref{eq:shape_thm} that $\#\partial_{\textnormal{e}} B(t) = o(t^d)$ as $t\to\infty$.
	
	Note that \eqref{eq:volume} holds without any moment assumption on $t_e$. One can obtain better upper bounds on $\#\partial_\textnormal{e} B(t)$ if we assume more about the distribution of $t_e$. We first state a result about the convergence rate to the limit shape \cite[Theorem~3.1]{A}. If $\prob(t_e=0)<p_c$ and $\E e^{\alpha t_e} <\infty$ for some $\alpha>0$, then there exist a constant $c>0$ such that almost surely,
	\begin{equation}
	\label{eq:alexander}
	(t-ct^{1/2} \log{t})\mathcal{B}\subseteq \bar{B}(t) \subseteq (t+ct^{1/2}\log{t})\mathcal{B} \quad \text{for all large $t$.}
	\end{equation}
	By counting the edges in the annulus $(t+2ct^{1/2}\log{t})\mathcal{B} \setminus (t- 2ct^{1/2}\log{t})\mathcal{B}$, one can then obtain for some $c_4>0$, almost surely, 
	\[
	\#\partial_\textnormal{e} B(t) \leq c_4 t^{d-1/2}\log{t} \quad \text{for all large }t.
	\]
	However, this type of bound should be far from optimal, because otherwise the boundary would occupy a positive fraction of the annulus, and this should not be true for most distributions. Therefore, a different method should be used to obtain a sharper bound. 
	
	In the physics literature, it is believed that the size of the boundary of first-passage-type growth clusters of volume $n$ should behave like $n^{\frac{d-1}{d}}$ (see for instance \cite{L, ZS}). Using the shape theorem, this corresponds to the relation $\#\partial_\textnormal{e} B(t) \sim t^{d-1}$. However, the only known rigorous result, which is proved in \cite{Bouch}, is an upper bound of the form $n^{1-\frac{1}{d(2d+5)+1}}$. Our main results below show that under a weak moment condition $\E Y<\infty$, where $Y$ is the minimum of $2d$ independent edge-weights, one almost surely has $\#\partial_\textnormal{e} B(t) \leq a t^{d-1}$ for most times $t$. However, under other conditions, the boundary may be larger, or infinite. Indeed, the combination of Theorems~\ref{thm:main} and \ref{thm:holes} shows that, roughly speaking, if $Y$ has exactly $1-\alpha$ moments $(\alpha>0)$ and a sufficiently regular distribution, then due to the presence of many small holes in $B(t)$, $\#\partial_\textnormal{e} B(t)$ is larger, of order $t^{d-1+\alpha}$. In contrast, for the exterior boundary (which does not count holes), we have a smooth bound $t^{d-1}$ regardless of the moment condition. All these results are under the assumption that there are not too many zero-weight edges; that is, $\mathbf{P}(t_e=0)<p_c$. If, on the other hand, $\mathbf{P}(t_e=0) \in (p_c,1)$, then one can argue that for all large $t$, one has $\#\partial_\textnormal{e} B(t) = \infty$ but $\#\partial_\textnormal{e}^{\textnormal{ext}} B(t)$ is bounded in $t$. The intermediate case, $\mathbf{P}(t_e=0)=p_c$, is more complicated because in two dimensions, even the growth rate of $B(t)$ depends on the distribution of $t_e$ \cite{DLW}, and in higher dimensions, the growth rate is unknown (and depends on whether there is an infinite cluster at the critical point in independent percolation, and this is a major open problem). For these reasons, we leave this critical case to further investigations.
	
	There are related Markovian growth models called the Eden model \cite{Eden} and the $1$-type Richardson model \cite{R}, and they are equivalent to certain FPP (site or bond) models with exponential weights. Using the memoryless property of the exponential distribution, one can prove that $\frac{\textnormal{d}}{\textnormal{d} t} \E\# B(t) = \E \#\partial_{\textnormal{e}} B(t)$, which implies $\int_0^t \E\#\partial_{\textnormal{e}} B(s)\;\textnormal{d}s \leq Ct^d$. That is, on average, $\E\#\partial_{\textnormal{e}} B(s) \sim s^{d-1}$.
	
	Throughout this article, we use $\Leb$ to denote the $1$-dimensional Lebesgue measure. For $a, b\in\R$, we define $a\wedge b = \min\{a,b\}$ and $a\vee b=\max\{a,b\}$. For $t \geq 0$, we define
	\begin{equation}
	\label{eq:box}
	S(t)=[-\lfloor t\rfloor, \lfloor t\rfloor]^d. 
	\end{equation}
	We write $\mathbf{e}_1$ for the first coordinate vector $(1, 0, \ldots, 0)$. Also, the symbols $C_i$, where $i$ is an integer, represent constants depending only on the dimension $d$ and the distribution of $t_e$. The same symbols $C_i$ will be used in different sections but they might possibly represent different numbers. 
	
	\subsection{Main results}
	\subsubsection{Rough times}
	Define 
	\[
	Y=\min\{t_1,\ldots,t_{2d}\}, 
	\]
	where $t_1,\ldots,t_{2d}$ are i.i.d. copies of $t_e$. For $a>0$ and $t>0$, we define sets of $a$-rough times as
	\[
	R_t(a) = \{s\in [0,t]: \#\partial_{\textnormal{e}} B(s)\geq as^{d-1}\E[Y\wedge s]\}
	\]
	and
	\[
	R^{\textnormal{ext}}_t(a) = \{s\in [0,t]: \#\partial^{\textnormal{ext}}_{\textnormal{e}} B(s)\geq as^{d-1}\}
	\]
	depending on which boundary we are discussing. Note that the definition of $R_t(a)$ includes an additional factor of $\mathbf{E}[Y \wedge s]$ in the lower bound, and its purpose is to allow for cases in which $\mathbf{E}Y = \infty$. Ignoring the term $\E[Y\wedge s]$ (assuming for the moment that this term is uniformly bounded in $s$), if one believes that $\#\partial_{\textnormal{e}} B(s)$ or $\#\partial^{\textnormal{ext}}_{\textnormal{e}} B(s)$ is of order $s^{d-1}$, then when $a$ is large, these sets represent times when the boundary is rough. Indeed, we will show that the upper density of the set of rough times is small when $a$ is large:
	\begin{thm}
		\label{thm:main}
		Suppose that $\prob(t_e=0)<p_c$.
		\begin{enumerate}
			\item[(a)] There exists $C>0$ such that almost surely,
			\[
			\limsup_{t\to\infty}\frac{\Leb(R_t(a))}{t} \leq \frac{C}{a}.
			\]
			\item[(b)] There exists $C>0$ such that almost surely,
			\[
			\limsup_{t\to\infty}\frac{\Leb(R^{\textnormal{ext}}_t(a))}{t} \leq \frac{C}{a}.
			\]
		\end{enumerate}
	\end{thm}
	
	\begin{rem}
		To understand the term $\E[Y\wedge t]$, let us consider the following cases:
		\begin{enumerate}
			\item If $\E Y<\infty$, then using $\E[Y\wedge t] \leq \E Y$, Theorem~\ref{thm:main}(a) says that $\#\partial_{\textnormal{e}} B(t)\leq a\E Y t^{d-1}$ for most $t$.
			\item If there exists a constant $C>0$ such that $\prob(Y\geq y)\leq C/y$ for all $y>0$, but $\E Y=\infty$, then
			\[
			\E[Y\wedge t] = \int_0^t \prob(Y\geq y)\;\textnormal{d}y \leq C'\log{t}
			\]
			when $t$ is large. In this case, Theorem~\ref{thm:main}(a) implies that $\#\partial_{\textnormal{e}} B(t)\leq at^{d-1}\log{t}$ for most $t$.
			
			\item Likewise, if we assume $\prob(Y\geq y)\leq C/y^{1-\alpha}$ for some $C>0$, $\alpha\in(0,1)$ and for all $y>0$, a similar calculation gives $\#\partial_{\textnormal{e}} B(t)\leq at^{d-1+\alpha}$ for most $t$.
		\end{enumerate}
	\end{rem}
	
	\subsubsection{Lower bound}	
	Here we present lower bounds for $\#\partial_{\textnormal{e}} B(t)$.
	
	\begin{thm}
		\label{thm:holes}
		Suppose that $\prob(t_e=0)<p_c$ and let $F_Y$ be the distribution function of $Y$. There exists $C>0$ such that almost surely,
		\[
		\#\partial_{\textnormal{e}} B(t) \geq C\left[(1-F_Y(t))\vee \frac{1}{t}\right]t^d \quad \text{for all large }t.
		\]
	\end{thm}
	
	\begin{rem}
		Similarly, to understand the term $1-F_Y(t)$, let us consider the following cases.
		\begin{enumerate}
			\item If $\E Y<\infty$, then by Markov's inequality, the order of $1-F_Y(t)$ is no larger than $1/t$ as $t\to\infty$, which in particular implies that $\#\partial_{\textnormal{e}} B(t)\geq Ct^{d-1}$.  This coincides with the upper bound from Theorem~\ref{thm:main}.
			\item If $\prob(Y\geq y)\geq C/y^{1-\alpha}$ for some $C>0$, $\alpha\in[0,1)$ and for all large $y>0$, then $\#\partial_{\textnormal{e}} B(t) \geq C't^{d-1+\alpha}$. In particular, if $C/y^{1-\alpha}\leq \prob(Y\geq y)\leq C'/y^{1-\alpha}$, then the upper and lower bounds for $\#\partial_\textnormal{e} B(t)$ match if $\alpha > 0$, and do not match when $\alpha =0$ because of a $\log$ factor.
		\end{enumerate}
	\end{rem}
	The previous two theorems show that under the condition $\mathbf{P}(t_e=0)<p_c$, one has upper and lower bounds of the form
	\[
	\left[ t(1-F_Y(t)) \vee 1\right] t^{d-1} \lesssim \#\partial_{\textnormal{e}} B(t) \lesssim \mathbf{E}[Y \wedge t] t^{d-1}.
	\]
	It is natural to ask how different these upper and lower bounds can be. From the above examples, we see that their ratio can be at least $\log t$. Below we will see that it can be made arbitrarily large (up to order $t$) infinitely often by choosing very irregular tails for the distribution of $t_e$. Yet for any distribution, we can also show that the ratio is at most $\log t$ for an unbounded set of $t$. To be precise, we claim the following:
	\begin{enumerate}
		\item[(a)] The ratio
		\begin{equation}\label{eq: nachos_bellegrande}
		\frac{\mathbf{E}[Y \wedge t]}{t(1-F_Y(t)) \vee 1}=o(t)
		\end{equation}
		as $t\to\infty$, but it can be made arbitrarily close to $t$ infinitely often. For instance, for any $k\geq 1$, we can find distributions such that the ratio is at least $Ct/\log\log\cdots\log{t}$ for an unbounded set of $t$, where we compose the $\log$ function $k$ times.
		\item[(b)] There is a constant $C>0$ such that for infinitely many $n$,
		\[
		\frac{\E [Y\wedge 2^n]}{2^n (1-F_Y(2^n)) \vee 1} \leq C\log(2^n).
		\]
	\end{enumerate}
	\begin{proof}[Proof of Claim]
		\begin{enumerate}
			\item[(a)] Note that by the bounded convergence theorem, as $t \to \infty$,
			\[
			0 \leq \frac{1}{t}\frac{\mathbf{E}[Y \wedge t]}{t(1-F_Y(t)) \vee 1} \leq \frac{1}{t}\mathbf{E}[Y \wedge t] = \E\left[\frac{Y}{t}\wedge 1\right] \to 0.
			\]
			
			For the second part, for simplicity we only show the case $k=1$ in detail. We inductively define a sequence $x_1=3$ and $x_{n+1}=x_n^{x_n}$ for all $n\in\mathbb{N}$. We then define a distribution for $t_e$ satisfying
			\[
			\prob(t_e > t) = (\log{x_n})^{-1/2d}
			\]
			if $t\in [x_{n-1}, x_n)$ and $n>1$ (and define $\prob(t_e>t)=1$ if $t<3$). Then for $t=x_{n-1}$ and $n>2$,
			\[
			t(1-F_Y(t)) = t\prob(Y>t) = (x_{n-1})(\log{x_n})^{-1} = (\log{x_{n-1}})^{-1} \leq 1,
			\]
			and
			\begin{align*}
			\E[Y\wedge t] \geq (x_{n-1}-x_{n-2})(\log{x_{n-1}})^{-1} & \geq \frac{1}{2} x_{n-1} (\log x_{n-1})^{-1}\\
			& = \frac{t}{2\log{t}}.
			\end{align*}
			So
			\[
			\frac{\mathbf{E}[Y \wedge t]}{t(1-F_Y(t)) \vee 1} = \E[Y \wedge t] \geq \frac{t}{2\log{t}}.
			\]
			
			Similarly, one can construct a distribution such that  for an unbounded set of $t$,
			\[
			\frac{\mathbf{E}[Y \wedge t]}{t(1-F_Y(t)) \vee 1} \geq \frac{Ct}{\log\log\cdots\log{t}}.
			\] 
			This can be done by considering a sequence of $x_n$'s that increases rapidly enough and replacing $\log{x_n}$ by $\log\log\cdots\log{x_n}$ in the above discussion.
			
			\item[(b)] There are two cases: either the sequence $(2^n\prob(Y > 2^n))$ is unbounded, or it is bounded. In the first case,
			\begin{align*}
			\E [Y\wedge 2^n] &\leq 1+ \sum_{k=1}^n \int_{2^{k-1}}^{2^
				k} \prob(Y > t) \;\text{d}t\\
			&\leq 1+ \sum_{k=1}^n 2^{k-1}\prob(Y > 2^{k-1}).
			\end{align*}
			Since the sequence $(2^n\prob(Y> 2^n))$ is unbounded, we can find infinitely many $n$ such that $2^n\prob(Y > 2^n) > 2^k \prob(Y > 2^k)$ for all $k=0,1,\ldots, n-1$. For all such large $n$, we have
			\[
			\E [Y\wedge 2^n] < 1+ n 2^n\prob(Y > 2^n) \leq 2\log(2^n) [2^n\prob(Y > 2^n)\vee 1].
			\]
			In the second case,
			\begin{align*}
			\E [Y\wedge 2^n] &\leq 1+ \sum_{k=1}^n 2^{k-1}\prob(Y > 2^{k-1})\\
			&\leq Cn\\
			&\leq C\log(2^n)[2^n\prob(Y > 2^n)\vee 1].
			\end{align*}
		\end{enumerate}
	\end{proof}
	
	Theorems~\ref{thm:main}(b) states that the edge exterior boundary of $B(t)$ is always small, while for certain heavy-tailed edge-weight distributions, Theorem~\ref{thm:holes} states that the full edge boundary is large. This means that there must be holes in $B(t)$. These holes cannot be too big, as one can argue by lattice animal arguments, so there must be many small holes. In fact, our proof of Theorem~\ref{thm:holes} shows that holes of size $1$ contribute a positive fraction to the full boundary in many low moment cases. It would be interesting to formally study the topology of $B(t)$ and its holes. 
	
	Last, we remark that analogous statements will hold if we replace $\#\partial_{\textnormal{e}} B(t)$ and $\#\partial^{\textnormal{ext}}_{\textnormal{e}}B(t)$ by $\E\#\partial_{\textnormal{e}} B(t)$ and $\E\#\partial^{\textnormal{ext}}_{\textnormal{e}}B(t)$ in Theorems~\ref{thm:main} and \ref{thm:holes}. The proofs are similar to those of Theorems~\ref{thm:main} and \ref{thm:holes}, and so we do not include them. 
	
	\subsubsection{Uniform curvature}
	We can even obtain that $\#\partial_{\textnormal{e}} B(t)$ is at most of order $(\log{t})^Ct^{d-1}$ for some $C>0$ in certain cases. Unfortunately, we will need to assume Newman's ``uniform curvature condition'' \cite{newman} which, although it is expected to be true for most edge-weight distributions, is unproved. For its statement, let $g$ be the norm on $\mathbb{R}^d$ whose unit ball is $\mathcal{B}$.
	
	\begin{df}\label{def: curvature}
		We say that $\mathcal{B}$ satisfies the uniform curvature condition if there are constants $C>0$, $\eta>1$ such that for all $z=\lambda z_1+(1-\lambda)z_2$ with $g(z_1)=g(z_2)=1$ and $\lambda\in[0,1]$,
		\[
		1-g(z) \geq C[\min\{g(z-z_1), g(z-z_2)\}]^\eta.
		\]
	\end{df}
	The following theorem states that if we assume that $\mathcal{B}$ satisfies the uniform curvature condition and that $t_e$ has finite exponential moments, then $\#\partial_{\textnormal{e}} B(t)$ is at most of order $(\log{t})^C t^{d-1}$ almost surely.
	
	\begin{thm}
		\label{thm:curv}
		Suppose that $\prob(t_e=0)<p_c$, $\E e^{\alpha t_e}<\infty$ for some $\alpha>0$ and $\mathcal{B}$ satisfies the uniform curvature condition. Then there exists $C>0$ such that almost surely for all large $t$, 
		\[
		\#\partial_{\textnormal{e}} B(t) \leq C (\log{t})^{C}t^{d-1}.
		\]
	\end{thm}
	
	It is not known if there exists a distribution such that $\mathcal{B}$ satisfies the uniform curvature condition. However, it is believed that this condition holds for $(t_e)$ having continuous distribution. See \cite[Section~2.8]{survey} for further discussion.

	\subsection{Sketch of proofs}
	\subsubsection{Theorem~\ref{thm:main}}
	To show Theorem~\ref{thm:main}(a), the idea is consider the amount of time $s \in [0,t]$ that an edge $e$ is on the boundary $\partial_\textnormal{e} B(s)$. It is not difficult to see that this amount of time is bounded above by $T(x,y) \wedge t$, if $e = \{x,y\}$. If $\mathbf{E}T(x,y) < \infty$, then on average, each edge is on the boundary for a constant amount of time. In this case, the ball $B(t)$ will grow by at least order $\#\partial_\textnormal{e} B(t)$ number of edges in a constant time. This means that if the boundary is too large for too long, then the growth of $B(t)$ will be so large as to violate the shape theorem.
	
	Formally, we consider the indicator $\mathbf{1}_{\{e\in\partial_{\textnormal{e}} B(s)\}}$, where $e\in\mathcal{E}^d$ and $s\in [0,t]$. If we fix an edge $e$ and integrate over $s$, we obtain the amount of time that $e = \{x,y\}$ stays on the boundary, which is bounded by $T(x,y) \wedge t$. Now, when we further sum over the edges in a box $[-t,t]^d$, we obtain an upper bound $\sum_{\{x,y\}} T(x,y)\wedge t$. Since there are $Ct^d$ many edges in the box $[-t,t]^d$, and the $T(x,y)$'s can be well-controlled by weakly-dependent random variables with the same tail properties as those of $Y$, we use Lemma~\ref{lem:truncation} to conclude that with high probability,
	\[
	\sum_{e \in [-t,t]^d} \int_0^t \mathbf{1}_{\{e \in \partial_\textnormal{e} B(s)\}}~\text{d}s \leq Ct^d\E[Y\wedge t].
	\]

	If we instead fix $s$ and sum over the edges first, we obtain $\#\partial_\textnormal{e} B(s)$, on the high probability event that $\partial_\textnormal{e} B(s) \subseteq [-t,t]^d$. Applying the above inequality, we obtain with high probability
	\[
	\frac{1}{t} \int_0^t \#\partial_\textnormal{e} B(s)~\text{d}s \leq Ct^{d-1} \mathbf{E}[Y\wedge t].
	\]
	In other words, the time-average of $\#\partial_\textnormal{e}B(s)$ is at most of order $s^{d-1} \mathbf{E}[Y\wedge s]$. Applying Lemma~\ref{lem:rough time} (the regularity lemma) will convert this integral inequality to the desired bound on the size of the set of rough times.
	
	For the edge exterior boundary, we are able to remove the term $\E[Y\wedge t]$ because of the following two facts:
	\begin{itemize}
			\item the exterior boundary of $B(t)$ forms a ``closed surface,'' and the number of such surfaces with cardinality $n$ is at most $e^{Cn}$;
			\item if $\alpha>0$ is large then for a deterministic such closed surface, the probability that more than a fixed constant fraction of its edges have edge-weights $>\alpha$ is at most $e^{-2Cn}$.
	\end{itemize}
	These will imply that when $t$ is large, at least a fixed constant fraction of the edges in $ \partial_\textnormal{e}^\textnormal {ext} B(t)$ have edge-weights at most $\alpha$. As in the case of $\partial_\textnormal{e} B(t)$, this means that such edges will be on the boundary for at most a constant amount of time. We conclude with an argument similar to the previous case (replacing $\mathbf{1}_{\{e \in \partial_\textnormal{e}B(s)\}}$ by $\mathbf{1}_{\{e\in\partial^{\textnormal{ext}}_{\textnormal{e}} B(s), t_e\leq\alpha\}}$).
	
	
	\subsubsection{Theorem~\ref{thm:holes}}
	To find a lower bound for $\#\partial_{\textnormal{e}} B(t)$, fix $\delta\in (0,1)$ and observe the following: if (a) all the edges adjacent to a vertex $x$ have edge-weights $> (1-\delta)t$, (b) $T(0,x-\mathbf{e}_1)\leq t$ and (c) $T(0,z)\geq \delta t$ for all $z$ such that $\|z-x\|_1=1$, then $\{x-\mathbf{e}_1,x\} \in \partial_{\textnormal{e}} B(t)$, since all the paths from $0$ to $x$ have passage time $>t$ but $T(0,x-\mathbf{e}_1)\leq t$. Such an edge $\{x-\mathbf{e}_1,x\}$ is surrounded by edges of high weight ($>(1-\delta)t$), but is adjacent to a vertex in $B(t)$. Therefore, $\#\partial_{\textnormal{e}} B(t)$ will be bounded below by the number of vertices $x$ satisfying all the above three conditions.
	
	Almost surely, when $\|x\|_1 \in [ct/2,ct]$ with $t$ large and $c$ fixed (but $\delta$ small), (b) and (c) are true. (c) can be shown using \cite[Proposition~5.8]{kesten2} (see also Lemma~\ref{lem:largedeviation}). (b) can easily be shown to hold if $t_e$ satisfies certain moment conditions, for instance having a finite exponential moment, but this is stronger than what we assume. We will instead use a coupling with Bernoulli bond percolation. Define an edge $e$ to be open if $t_e\leq M$, where $M$ is sufficiently large to ensure that $\prob(t_e\leq M) > p_c$. It is known (from Antal-Pisztora \cite{AP}) that in supercritical bond percolation, the distance in the infinite open cluster $\mathsf{C}$ is bounded above by a constant times the $\ell^\infty$-distance with high probability. Using this, one can show that if $x\in\mathsf{C}$ and $\|x\|_\infty$ is sufficiently large, then $T(0,x)\leq CM\|x\|_\infty$. Therefore (b) holds with high probability so long as $c$ is small and $x-\mathbf{e}_1 \in \mathsf{C}$. See Lemma~\ref{lem:upperbound} for more details. 
	
	Therefore, it suffices to lower bound the number of vertices $x$ with $x-\mathbf{e}_1 \in \mathsf{C}$ that satisfy (a). By the ergodic theorem, there is a positive density of $x \in \mathbb{Z}^d$ such that $x-\mathbf{e}_1 \in \mathsf{C}$. If we take such an $x$ and artificially raise the edge-weights of edges incident to $x$ to be larger than $(1-\delta)t$, then, so long as $x-\mathbf{e}_1$ is still in $\mathsf{C}$ after the modification, we will have a vertex $x$ with the required properties. The total probability cost of this operation is of order $1-F_Y(t)$, and so the expected number of such vertices in a box $[-t,t]^d$ should be of order $t^d(1-F_Y(t))$. If we combine this bound with the lower bound of \eqref{eq:volume}, we obtain the desired result. To rigorously perform this modification, we use a shielding lemma, which is given as Lemma~\ref{lem: independence}, and to move from the expected number of such vertices to an almost sure bound, we apply Bernstein's inequality, stated as Theorem~\ref{thm:bernstein}.

	
	\subsubsection{Theorem~\ref{thm:curv}}
	We would like to show that $\#\partial_{\textnormal{e}} B(t) \leq C (\log{t})^{C}t^{d-1}$ under the uniform curvature assumption. The idea is to cover $B(t)$ by at most order $t^{d-1}$ many sectors of volume order $t$, and show that each sector can contain at most $(\log{t})^C$ many edges from $\partial_e B(t)$. 
	
	To estimate the number of edges in a sector that are on $\partial_e B(t)$, note that if $e = \{u,v\}$ is in $\partial_e B(t)$ with $u \in B(t)$, then
	\[
	t < T(0,v) \leq T(0,u) + t_e \leq t+t_e.
	\]
	Under our exponential moment condition, with high probability, all edges in $\partial_eB(t)$ can be shown to have weight at most $(\log t)^C$, so we obtain $|T(0,u) - t| \leq (\log t)^C$. If $f = \{w,z\}$ is another edge in $\partial_e B(t)$ with $w \in B(t)$, then
	\[
	|T(0,u) - T(0,w)| \leq 2(\log t)^C.
	\]
	In other words, the passage times from the origin to endpoints of different edges on the boundary must be within a power of $\log t$ of each other.
	
	Because of the small aperture of our sectors, if there are edges $e,f$ in one sector in $\partial_e B(t)$, then they lie close to some ray of the form $\{sx : s \geq 0\}$, where $x$ is a unit vector. Therefore we can find $k \geq \ell$ such that $kx$ is close to $e$ and $\ell x$ is close to $f$, and $|T(0,\ell x) - T(0,kx)| \geq c(\log t)^C$. However, in Proposition~\ref{prop: unif_curv}, we prove that there is a constant $C>0$ such that for any $x\in \R^d$ with $\|x\|_2=1$ and for any $k \geq \ell$, one has with high probability
	\begin{equation}
	\label{eq: busemann}
	T(0,kx) - T(0,\ell x) \geq C(k-\ell).
	\end{equation}
	This inequality implies that our $k$ and $\ell$ above must be at most order $(\log t)^C$ distance from each other. In other words, the intersection of $\partial_eB(t)$ with the sector associated to the ray has size at most order $(\log t)^C$, and this would complete the proof.

%
%
	
	
	To show \eqref{eq: busemann} holds with high probability, we use techniques developed by Newman to control geodesic (optimal path) ``wandering'' under the uniform curvature assumption. With high probability, the optimal path from $kx$ to $0$ can be shown to come within distance $(k-\ell)^c$ of $\ell x$, where $c < 1$. (See \eqref{eq: taco_2}, where $M=(k-\ell)/2$, and Figure~\ref{fig: fig_2}.) If $y$ is a point of this path that is close to $\ell x$, then
	\[
	T(0,kx)-T(0,\ell x) \geq T(kx,y) - T(\ell x,y) \ge C\left[ (k-\ell) - (k-\ell)^c\right] \geq C(k-\ell).
	\]


	\section{Preliminary results}
	The first tool we will need is the Cox-Durrett shape theorem \cite{CD}, which is stronger than \eqref{eq:shape_thm}.
	\begin{thm}[Shape theorem]
		Suppose that $\prob(t_e=0)<p_c$ and $\E Y^d<\infty$. There exists a nonrandom, compact, convex set $\mathcal{B}\subseteq\R^d$ with nonempty interior, such that for all $\varepsilon>0$, with probability $1$,
		\[
		(1-\varepsilon)\mathcal{B} \subseteq \frac{\bar{B}(t)}{t} \subseteq (1+\varepsilon)\mathcal{B} \quad \text{for all large } t.
		\]
	\end{thm}
	
	As a consequence of the shape theorem, one can show that even without the condition $\mathbf{E}Y^d<\infty, B(t)$ cannot grow too quickly.
	\begin{lem}
		\label{lem:weakshape1}
		Suppose that $\prob(t_e=0)<p_c$. Then there exists $M>0$ such that with probability $1$, $\bar{B}(t)\subseteq S(Mt)$ for all large $t$, where $S$ is defined in \eqref{eq:box}.
	\end{lem}
	\begin{proof}
		If we define $t_e'$ by $t_e' = t_e\wedge 1$, then for all $t\geq 0$, $B(t)\subseteq B'(t)$, where $B'(t)$ is the $T$-ball using weights $(t_e')$. Then $\bar{B}(t)\subseteq \bar{B'}(t)$ and applying the shape theorem for $(t_e')$ establishes the lemma.
	\end{proof}
	
	We will also need the following result of Kesten \cite[Proposition~5.8]{kesten2}.
	\begin{prop}
		If $\prob(t_e=0)<p_c$, then there exist $D_1, D_2, D_3>0$ such that for all $n \geq 1$, one has
		\[
		\prob(\textnormal{there exists a self-avoiding path $\gamma$ from $0$ with $\#\gamma\geq n$ but $T(\gamma)<D_1n$}) \leq D_2 e^{-D_3n}.
		\]
	\end{prop}
	From this, we immediately obtain a lower bound on $T$.
	\begin{lem}
		\label{lem:largedeviation}
		If $\prob(t_e=0)<p_c$, then for any $z\in\Z^d$,
		\[
		\prob(T(0,z) < D_1 \|z\|_1) \leq D_2 e^{-D_3\|z\|_1}.
		\]
	\end{lem}
	
	We now state some results from percolation theory that will be used in the proof of Theorem~\ref{thm:holes}. For $p\in[0,1]$, let $\prob_p = \prod_{e\in\mathcal{E}^d} \mu_e$ be the product measure on $\{0,1\}^{\mathcal{E}^d}$, where each $\mu_e = p\delta_1+(1-p)\delta_0$. We say that an edge $e$ is open if $\omega(e)=1$, where $\omega$ is a typical element of the sample space $\{0,1\}^{\mathcal{E}^d}$. It is known that when $p>p_c$, there almost surely exists a unique infinite open cluster (that is, the subgraph induced by the open edges has an infinite connected component) \cite[Theorem~8.1]{Grimmett}. We denote by $\mathsf{C}$ the infinite open cluster and write $\textnormal{dist}_{\mathsf{C}}$ for the (graph) distance in $\mathsf{C}$.
	\begin{thm}[Theorem~1.1 \cite{AP}]
		\label{thm:antalpisztora}
		Let $p>p_c$. Then there exist $p$-depending constants $D_4, D_5 >0$ such that
		\[
		\prob_p(\textnormal{dist}_\mathsf{C} (x,y)\geq D_4\|x-y\|_\infty, x,y\in \mathsf{C})\leq e^{-D_5\|x-y\|_\infty}
		\]
		for all $x,y\in\Z^d$.
	\end{thm}
	These results in Bernoulli bond percolation allow us to upper bound $T(0,x)$ if $M$ is large and $x$ is in the infinite open cluster, where we say an edge $e$ is open if $t_e\leq M$.
	\begin{lem}
		\label{lem:upperbound}
		Suppose that $\prob(t_e=0)<p_c$. Fix $M>0$ such that $\prob(t_e\leq M)>p_c$. Define a percolation configuration $(\omega(e))$ by $\omega(e) = \mathbf{1}_{\{t_e\leq M\}}$. Then
		\[
		\prob(T(0,x) \geq 4D_4M\|x\|_\infty \textnormal{ for infinitely many $x\in\mathsf{C}$}) =0,
		\]
		where $D_4$ is as in Theorem~\ref{thm:antalpisztora}.
	\end{lem}
	\begin{proof}
		Let $\varepsilon>0$. Let $A_k$ be the event that $S(k)$ intersects $\mathsf{C}$. Since $\mathsf{C}$ exists and is unique almost surely, we can fix $k\in\mathbb{N}$ such that $\prob(A_k)>1-\varepsilon$. 
		
		We decompose the event in the statement of the lemma as
		\begin{align*}
		&\phantom{\subseteq}\{T(0,x) \geq 4D_4M\|x\|_\infty \textnormal{ for infinitely many $x\in\mathsf{C}$}\}\\
		&\subseteq A_k^c \cup  \left(A_k \cap \{ T(0,x) \geq 4D_4M\|x\|_\infty \textnormal{ for infinitely many $x\in\mathsf{C}$}\} \right)\\
		&\subseteq A_k^c \cup B_1 \cup B_2,
		\end{align*}
		where
		\[
		B_1 = \{T(y,x) \geq 2D_4M\|x\|_\infty \textnormal{ for some $y\in S(k)\cap\mathsf{C}$ and for infinitely many $x\in\mathsf{C}$}\}
		\]
		and
		\[
		B_2 = \{T(0,y) \geq 2D_4M\|x\|_\infty \textnormal{ for some $y\in S(k)\cap\mathsf{C}$ and for infinitely many $x\in\mathsf{C}$}\}
		\]
		The event $A_k^c$ has probability at most $\varepsilon$ and $B_2$ almost surely does not occur. For $B_1$, if $\|x\|_\infty$ is sufficiently large (for instance $\|x\|_\infty \geq k$), then $T(y,x)\geq 2D_4M\|x\|_\infty$ implies $T(y,x) \geq D_4M\|x-y\|_\infty$. When $x,y\in \mathsf{C}$, this implies $\textnormal{dist}_\mathsf{C}(x,y) \geq D_4\|x-y\|_\infty$. By Theorem~\ref{thm:antalpisztora} and a union bound, we see that $\prob(B_1) = 0$.
		
		Therefore, we have
		\[
		\prob(T(0,x) \geq 4D_4M\|x\|_\infty \textnormal{ for infinitely many $x\in\mathsf{C}$}) \leq \varepsilon.
		\]
		Since $\varepsilon>0$ is arbitrary, this completes the proof.
	\end{proof}
	
	Finally, we need Bernstein's inequality \cite[Eq.~(2.10)]{BLM}. We state the inequality here for the reader's convenience.
	\begin{thm}[Bernstein's inequality]
		\label{thm:bernstein}
		If $X_1,\ldots, X_n$ are independent with $|X_i|\leq b$  almost surely for all $i$, then for all $t\geq 0$,
		\[
		\prob(|X_1+\cdots+X_n - \E(X_1+\cdots+X_n)|\geq t)\leq 2\exp\left(-\frac{t^2}{2(bt/3+\sum_{i=1}^n\E X_i^2)}\right).
		\]
	\end{thm}
	
	\section{Proofs of Theorems}	
	\subsection{Proof of Theorem~\ref{thm:main}}
	To show Theorem~\ref{thm:main}, we need the following lemma, which will be used to give an estimate on the frequency of rough times. It is a form of Markov's inequality for functions defined on the real line.
	\begin{lem}[Regularity lemma]
		\label{lem:rough time}
		Let $C, s_0>0$ be constants. Let $\phi,\psi:[0,\infty)\to[0,\infty)$ be Lebesgue measurable functions such that
		\begin{enumerate}
			\item $\int_0^t \phi(s)\;\textnormal{d}s \leq Ct^d\psi(t)$ for all $t> s_0$;
			\item $\psi$ is nondecreasing with $0<\psi(2t)\leq 2\psi(t)$ for all $t>0$.
		\end{enumerate} 
		Then for $t > 0$, one has
		\[
		\frac{\Leb(\{s\in[0,t]: \phi(s)\geq as^{d-1}\psi(s)\})}{t} \leq \frac{2s_0}{t} + \frac{2^{d+1}C}{a}.
		\]
	\end{lem}
	\begin{proof}
		We may assume that $t>2s_0$. Let $i_0 \geq 1$ be such that
		\[
		\frac{t}{2^{i_0+1}} \leq s_0 < \frac{t}{2^{i_0}}.
		\]
		
		For $i=0,1,2,\ldots$, define $t_i = t/2^i$ and 
		\[
		R_i = \{s\in[t_{i+1}, t_i]: \phi(s)\geq as^{d-1}\psi(s)\}.
		\]
		If $i < i_0$ (so that $t_i > s_0$), then
		\begin{align*}
		Ct_i^d\psi(t_i) &\geq \int_0^{t_i} \phi(s)\;\text{d}s\\
		&\geq \int_{R_i} \phi(s)\;\text{d}s\\
		&\geq \int_{R_i} as^{d-1}\psi(s)\;\text{d}s\\
		&\geq at_{i+1}^{d-1}\psi(t_{i+1})\Leb(R_i), 
		\end{align*}
		which implies 
		\[
		\Leb(R_i) \leq \frac{Ct_i^d\psi(t_i)}{at_{i+1}^{d-1}\psi(t_{i+1})} \leq \frac{2^dC}{a}t_i.
		\]
		Summing over $i$ completes the proof:
		\[
		\Leb(\{s\in[0,t]: \phi(s)\geq as^{d-1}\psi(s)\}) = \sum_{i = 0}^\infty \Leb(R_i) \leq 2s_0 + \frac{2^{d+1}Ct}{a}.
		\]
	\end{proof}
	
	\subsubsection{Edge boundary}
	We first need a lemma that gives the asymptotic behavior of truncated random variables.
	\begin{lem}
		\label{lem:truncation}
		Let $Y_1,Y_2,\ldots$ be a sequence of i.i.d. nonnegative random variables and let $(C_n)$ be a sequence of numbers such that $0\leq C_n\leq c$ for some $c>0$ and for all $n$. For each $n\in\mathbb{N}$ and $i\leq n$, define $Z_i^{(n)} = Y_i\wedge C_nn^{1/d}$. Then almost surely, $\sum_{i=1}^n Z_i^{(n)} \leq 2n\E Z_1^{(n)}$ for all large $n$.
	\end{lem}
	\begin{proof}
		If $Y_1=0$ almost surely, the statement is trivial, so we suppose that $Y_1>0$ with positive probability. Then $\E Z_1^{(n)}>0$ for all $n\geq 1$, and by Theorem~\ref{thm:bernstein}, one has
		\[
		\prob\left(\left|\sum_{i=1}^n Z_i^{(n)} - n \E Z_1^{(n)}\right|\geq n\E Z_1^{(n)}\right) \leq 2\exp\left(-\frac{n^2(\E Z_1^{(n)})^2}{2(cn^{1+1/d}\E Z_1^{(n)}/3 + n \E (Z_1^{(n)})^2)}\right).
		\]
		Now, $\E (Z_1^{(n)})^2 \leq cn^{1/d} \E Z_1^{(n)}$, so this  is further bounded above by
		\[
		2\exp\left(-\frac{n^2(\E Z_1^{(n)})^2}{2(cn^{1+1/d}\E Z_1^{(n)}/3 + cn^{1+1/d} \E Z_1^{(n)})}\right) = 2\exp\left(-\frac{3n^{1-1/d}\E Z_1^{(n)}}{8c}\right).
		\]
		Since $\E Z_1^{(n)}>0$ and is bounded away from $0$,  the right side is summable in $n$. By the Borel-Cantelli lemma, almost surely, for all large $n$,
		\[
		\left|\sum_{i=1}^n Z_i^{(n)} - n \E Z_1^{(n)}\right|< n\E Z_1^{(n)}.
		\] 
		This proves Lemma~\ref{lem:truncation}.
	\end{proof}
	
	We are now ready to prove Theorem~\ref{thm:main}(a).
	\begin{proof}[Proof of Theorem~\ref{thm:main}(a)]
		The following arguments fall under the purview of the ``array method.'' For $e\in\mathcal{E}^d$ and $t \geq 0$, define
		\[
		f(e,t) = \mathbf{1}_{\{e\in\partial_{\textnormal{e}} B(t)\}}.
		\]
		
		On the one hand, for $e=\{x,y\}$, we have
		\begin{align*}
		\int_{0}^t f(e,s)\;\text{d}s &= \Leb(\{s\in[0,t]: e\in\partial_{\textnormal{e}} B(s)\})\\
		&\leq T(x,y)\wedge t,
		\end{align*}
		because the amount of time that the edge $e$ stays on the boundary is bounded above by $|T(0,x)-T(0,y)| \leq T(x,y)$. Write $E(S(Mt))$ for the set of edges with both endpoints in $S(Mt)$. Summing over $e \in E(S(Mt))$ yields
		\begin{equation}
		\label{eq:sumintbound}
		\sum_{e \in E(S(Mt))} \int_0^t f(e,s)\;\text{d}s \leq \sum_{e \in E(S(Mt))} [T(x,y)\wedge t].
		\end{equation}
		
		We claim that there exists a nonrandom constant $c_0>0$ such that almost surely, for all large $t$,
		\begin{equation}
		\label{eq:upperbound}
		\sum_{e \in E(S(Mt))} [T(x,y)\wedge t] \leq c_0t^{d} \E[Y\wedge t].
		\end{equation}
		We now show \eqref{eq:upperbound} and, from now on, we will write ``i.o.'' to mean ``for an unbounded set of  $t$.'' By dividing the sum into sparser ones and using a union bound and translation invariance, we find a nonrandom constant $C_d$, depending only on $d$, such that for any $\lambda>0$,
		\begin{align*}
		\prob\left(\sum_{e=\{x,y\} \in E(S(Mt))} [T(x,y)\wedge t]\geq \lambda t^{d} \E[Y\wedge t]\text{ i.o.}\right) \\
		\leq C_d \prob\left(\sum_{x\in 5\Z^d\cap S(Mt)} [T(x,x+\mathbf{e}_1)\wedge t] \geq \lambda t^{d}\E[Y\wedge t]/C_d\text{ i.o.}\right).
		\end{align*}
		Now, note we can construct $2d$ edge-disjoint (deterministic) paths $\gamma_1$, $\ldots$, $\gamma_{2d}$ from $0$ to $\mathbf{e}_1$ such that if $x,y\in 5\Z^d$ and $x\neq y$, then the paths $x+\gamma_1,\ldots,x+\gamma_{2d}$ and the paths $y+\gamma_1,\ldots,y+\gamma_{2d}$ are edge-disjoint. For $x\in\Z^d$, let $\tau_x$ be the minimum of the passage times of these $2d$ disjoint paths from $x$ to $x+\mathbf{e}_1$. Then the second term in the last inequality is further bounded above by
		\[
		C_d \prob\left(\sum_{x\in 5\Z^d\cap S(Mt)} (\tau_x\wedge t) \geq \lambda t^{d}\E[Y\wedge t]/C_d\text{ i.o.}\right).
		\]
		Now, from the proof of \cite[Lemma~3.1]{CD}, there exists another dimension-dependent constant $C_d'$ such that $\E[Y\wedge t]\geq C_d'\E[\tau_0 \wedge t]$ for all $t \geq 0$. Furthermore, for $t \geq 1$, $\mathbf{E}[\tau_0 \wedge t] \geq (1/2) \mathbf{E}[\tau_0 \wedge \lceil t \rceil]$. So we obtain a further upper bound
		\[
		C_d \prob\left(\sum_{x\in 5\Z^d\cap S(Mt)} (\tau_x\wedge \lceil t\rceil) \geq \lambda C_d'(\lfloor t\rfloor)^{d}\E[\tau_0\wedge \lceil t\rceil]/(2C_d)\text{ i.o.}\right).
		\]
		
		Observe that the $\tau_x$'s in the above sum are i.i.d.. By Lemma~\ref{lem:truncation} with $Y_i=\tau_x$, $n=\#(5\Z^d\cap S(Mt))$ and $C_n = \lceil t\rceil/n^{1/d}$, there exists $\lambda_0>0$ such that
		\[
		\prob\left(\sum_{x\in 5\Z^d\cap S(Mt)} (\tau_x\wedge \lceil t\rceil) \geq \lambda_0 (\lfloor t\rfloor)^{d}\E[\tau_0\wedge \lceil t\rceil]\text{ i.o.}\right) = 0.
		\]
		Hence, following the string of inequalities, we see that there exists $c_0>0$ such that
		\begin{align*}
		\prob\left(\sum_{e=\{x,y\} \in E(S(Mt))} [T(x,y)\wedge t]\geq c_0t^{d}\E[Y\wedge t]\text{ i.o.}\right)=0
		\end{align*}
		and this proves the claim, equation \eqref{eq:upperbound}.
		
		Combining \eqref{eq:sumintbound} and \eqref{eq:upperbound}, we find with probability $1$,
		\[
		\sum_{e\in E(S(Mt))} \int_{0}^t f(e,t)\;\text{d}t \leq c_0t^{d}\E[Y\wedge t] \quad  \text{for all large }t.
		\]
		By Lemma~\ref{lem:weakshape1}, let $M>0$ such that almost surely there exists a random $T_0>0$ such that for all $t>T_0$, $\partial_\textnormal{e} B(t)\subseteq E(S(Mt))$. On the event $\{T_0 < t\}$, since $\partial_{\textnormal{e}} B(t)\subseteq E(S(Mt))$, one has for all $s\in[0,t]$,
		\[
		\sum_{e\in E(S(Mt))} f(e,s) = \#\partial_{\textnormal{e}} B(s),
		\]
		and hence
		\[
		\int_{0}^t \sum_{e\in E(S(Mt))} f(e,s)\;\text{d}s = \int_{0}^t \#\partial_{\textnormal{e}} B(s) \;\text{d}s.
		\]
		Therefore, with probability $1$,
		\[
		\int_{0}^t \#\partial_{\textnormal{e}} B(s) \;\text{d}s \leq c_0t^{d}\E[Y\wedge t] \quad \text{for all large } t.
		\]
		Note that $\E[Y\wedge 2t] = 2\E[(Y/2)\wedge t] \leq 2\E[Y\wedge t]$. Taking $t \to \infty$ in Lemma~\ref{lem:rough time} with $\phi(t) = \#\partial_{\textnormal{e}} B(t)$ and $\psi(t) = \E[Y\wedge t]$, we obtain that almost surely,
		\[
		\limsup_{t\to\infty}\frac{\Leb(R_t(a))}{t} \leq \frac{2^{d+1}c_0}{a}.
		\]
	\end{proof}
	
	\subsubsection{Edge exterior boundary}
	In the course of the proof of Theorem~\ref{thm:main}(b), we will need the following purely graph-theoretic fact. Recall that a set $U \subseteq \Z^d$ is called $*$-connected if for each pair $u, v \in U$ there is a sequence $(u = w_0, w_1, \ldots, w_k = v)$ where each $w_i \in U$ and $\|w_i - w_{i+1}\|_{\infty} \leq 1$.
	\begin{lem}[\cite{T}, Lemma 2]\label{lem:timar}
		Let $V \subseteq \Z^d$ be finite and connected. Then $\partial^{\textnormal{ext}}V$ is $*$-connected.
	\end{lem}
	
	We will rule out the possibility that the edge exterior boundary of $B(t)$ contains too many large-weight edges, where ``large'' is relative to the distribution of $t_e$. To this end, let $\alpha > 0$ be large (to be chosen later so that $\prob(t_e > \alpha)$ is sufficiently small). We will say that a finite vertex set $W \subseteq \Z^d$ is an ``$\alpha$-bad contour'' if
	\begin{enumerate}
		\item $W$ is $*$-connected;
		\item $W$ encloses $0$ -- that is, any vertex-self-avoiding infinite $\Z^d$ path beginning at $0$ must contain a vertex of $W$;
		\item letting $W_\alpha :=  \{w \in W: \exists e \ni w \text{ with } t_e > \alpha \}$,
		we have $\# W_\alpha \geq \# W / 2$.
	\end{enumerate}
	Note only condition 3 involves the realization of the edge-weights.
	
	\begin{prop}
		\label{prop:badcontour}
		If $\alpha>0$ is sufficiently large, then there exists $C>0$, depending only on $\alpha$ and $d$, such that for all $n \geq 1$, 
		\[
		\prob(\textnormal{there exists an $\alpha$-bad contour of cardinality $\geq n$})\leq e^{-Cn}.
		\]
	\end{prop}
	To prove Proposition~\ref{prop:badcontour}, we first prove the following lemma, which gives an upper bound on the number of contours around $0$. It is a basic bound on lattice animals, like \cite[Eq.~(4.24)]{Grimmett}.
	\begin{lem}
		\label{lem:number_of_contours}
		For $n\in\mathbb{N}$, let $\mathcal{C}_n$ be the set of all $*$-connected $W\subseteq \Z^d$ such that $\# W=n$ and $W$ encloses $0$. Then 
		\[
		\#\mathcal{C}_n\leq n\left[\frac{(3^d)^{3^d}}{(3^d-1)^{3^d-1}}\right]^n.
		\]
	\end{lem}
	\begin{proof}[Proof of Lemma~\ref{lem:number_of_contours}]
		Define $\tilde{\mathcal{C}}_n$ to be the set of all $*$-connected sets $W\subseteq\Z^d$ with $\# W=n$ and $0\in W$. If $W\in\mathcal{C}_n$, then there exists $k\in\{0,1,\ldots,n-1\}$ such that $W-k\mathbf{e}_1\in \tilde{\mathcal{C}}_n$, and hence $\#\mathcal{C}_n\leq n\#\tilde{\mathcal{C}}_n$.
		
		To bound $\#\tilde{\mathcal{C}}_n$, we consider the measure $\prob_p' = \prod_{x\in\Z^d}\mu_x$ on the space $\{0,1\}^{\Z^d}$, where $p\in[0,1]$ and each $\mu_x = p\delta_1 + (1-p)\delta_0$ is the Bernoulli measure on $\{0,1\}$. We will write the elements in $\{0,1\}^{\Z^d}$ as $(\omega(x))_{x\in\Z^d}$. We say that $W\subseteq \Z^d$ is the $*$-open cluster of $0$ if $W$ is the maximal $*$-connected subset containing $0$ with $\omega(x)=1$ for all $x\in W$.	We also define the $*$-vertex boundary $\partial^* V$ of a bounded $V\subseteq \Z^d$ to be the set of all $x\in\Z^d\setminus V$ such that $\|x-y\|_\infty=1$ for some $y\in V$. Note that
		\begin{align*}
		1&\geq \prob_p'(\text{the  $*$-open cluster of $0$ has cardinality $n$})\\
		&= \sum_{W\in\tilde{\mathcal{C}}_n} p^{\#W}(1-p)^{\#\partial^* W}.
		\end{align*}
		Note that for finite $V\subseteq \Z^d$, each $x\in V$ has at most $3^d-1$ many distinct $*$-adjacent vertices on $\partial^* V$, and each vertex on $\partial^* V$ is adjacent to some $x\in V$. Thus $\#\partial^* V\leq (3^d-1)\# V$ and we have
		\[
		1\geq \sum_{W\in\tilde{\mathcal{C}}_n} p^{\#W} (1-p)^{(3^d-1)\#W} = \#\tilde{\mathcal{C}}_n[p(1-p)^{3^d-1}]^{n}.
		\]
		This inequality holds for all $p\in[0,1]$, so setting $p=3^{-d}$ yields
		\[
		\#\tilde{\mathcal{C}}_n \leq \left[\frac{(3^d)^{3^d}}{(3^d-1)^{3^d-1}}\right]^n
		\]
		and finishes the proof.
	\end{proof}
	\begin{proof}[Proof of Proposition~\ref{prop:badcontour}]
		Note that
		\begin{align*}
		&\prob(\text{there exists an $\alpha$-bad contour of cardinality $\geq n$}) \\
		&\leq \sum_{k=n}^\infty \prob(\text{there exists an $\alpha$-bad contour of cardinality $=k$})\\
		&\leq \sum_{k=n}^\infty \sum_{W\in\mathcal{C}_k} \prob(\text{$W$ is an $\alpha$-bad contour}).
		\end{align*}
		Fix $W\in\mathcal{C}_k$. If $W$ is an $\alpha$-bad contour, then at least $k/2$ many vertices of $W$ are in $W_\alpha$. Among these vertices, at least half of them are in $\Z_\textnormal{even}^{d}$ or at least half of them are in $\Z^d\setminus \Z_\textnormal{even}^{d}$, where $\mathbb{Z}_\textnormal{even}^{d} = \{x \in \mathbb{Z}^d : \|x\|_1 \text{ is even}\}$.
		
		For fixed $W$, the events $\{x\in W_\alpha\}$ for $x\in\Z_\textnormal{even}^{d}$ are independent (similarly for $\Z^d \setminus \Z_\textnormal{even}^{d}$). Writing $F$ the distribution function of $t_e$, we obtain
		\begin{align*}
		\prob(\text{$W$ is an $\alpha$-bad contour}) &= \sum_{m=\lceil k/2 \rceil}^k \mathbf{P}(\#W_\alpha = m) \\
		&\leq \sum_{m=\lceil k/2 \rceil}^k \sum_{W \ni x_1, \ldots, x_m \text{ distinct}} \mathbf{P}(\text{all } x_i \in W_\alpha)\\
		&\leq \sum_{m=\lceil k/2\rceil}^k {k\choose m} (1-F(\alpha)^{2d})^{m/2}\\
		&\leq (1-F(\alpha)^{2d})^{k/4}\sum_{m=\lceil k/2\rceil}^k {k\choose m}\\
		&\leq 2^k(1-F(\alpha)^{2d})^{k/4}.
		\end{align*}
		
		Therefore
		\begin{align*}
		&\prob(\text{there exists an $\alpha$-bad contour of cardinality $\geq n$})\\
		& \leq \sum_{k=n}^\infty k\left[\frac{(3^d)^{3^d}}{(3^d-1)^{3^d-1}}\right]^k 2^{k}(1-F(\alpha)^{2d})^{k/4}\\
		&\leq e^{-Cn}
		\end{align*}
		for all $n$, if $\alpha$ is sufficiently large.
	\end{proof}
	We can now prove Theorem~\ref{thm:main}(b).
	\begin{proof}[Proof of Theorem~\ref{thm:main}(b)]
		By Lemma~\ref{lem:weakshape1}, let $M>0$ be such that there exists a random $T_0>0$ such that for all $t>T_0$, $\partial^{\textnormal{ext}}_{\textnormal{e}} B(t)\subseteq E(S(Mt))$. Fix $\alpha>0$ such that the conclusion of Proposition~\ref{prop:badcontour} holds. By the Borel-Cantelli lemma and Lemma~\ref{lem:timar}, together with the fact that almost surely $\#\partial^{\textnormal{ext}}_{\textnormal{e}} B(t)\to\infty$ as $t\to\infty$, there exists a random $T_1\geq T_0$ such that for all $t>T_1$, $\partial^{\textnormal{ext}} B(t)$ is not an $\alpha$-bad contour.
		
		For $e\in\mathcal{E}^d$ and $t \geq 0$, define
		\[
		h(e,t) = \mathbf{1}_{\{e\in\partial^{\textnormal{ext}}_{\textnormal{e}} B(t), t_e\leq\alpha\}}.
		\]
		Consider an outcome in the event $\{t>T_1\}$. For any $e\in\mathcal{E}^d$,
		\[
		\int_0^t h(e,s)\;\textnormal{d}s \leq \Leb\{t \geq 0 : e \in \partial_{\textnormal{e}} B(t)\} \mathbf{1}_{\{t_e \leq \alpha\}} \leq t_e \mathbf{1}_{\{t_e \leq \alpha\}} \leq \alpha,
		\]
		and hence
		\begin{equation}
		\label{eq:sumintbound2}
		\sum_{e \in E(S(Mt))} \int_0^t h(e,s)\;\textnormal{d}s \leq C_d \alpha M^dt^d.
		\end{equation}
		
		On the other hand, since $\partial^{\textnormal{ext}}_{\textnormal{e}} B(t)\subseteq E(S(Mt))$, for any $s\in[0,t]$,
		\[
		\sum_{e \in E(S(Mt))} h(e,s) = \#\{e\in \partial^{\textnormal{ext}}_{\textnormal{e}} B(s): t_e\leq \alpha\}.
		\]
		For $s>T_1$, $\partial^{\textnormal{ext}} B(s)$ is not an $\alpha$-bad contour, so for $s$ with $T_1 < s \leq t$,
		\[
		\sum_{e \in E(S(Mt))} h(e,s) \geq \frac{1}{2}\#\partial^{\textnormal{ext}} B(s) \geq \frac{1}{4d} \#\partial^{\textnormal{ext}}_{\textnormal{e}} B(s).
		\]
		Therefore, on the event $\{t>T_1\}$,
		\[
		\int_{T_1}^t \#\partial^{\textnormal{ext}}_{\textnormal{e}} B(s)\;\textnormal{d}s \leq 4d \int_0^t\sum_{e \in E(S(Mt))} h(e,s)\;\textnormal{d}s.
		\]
		Combining this with \eqref{eq:sumintbound2}, we have
		\[
		\int_0^t \#\partial^{\textnormal{ext}}_{\textnormal{e}} B(s)\;\textnormal{d}s \leq \int_0^{T_1} \#\partial^{\textnormal{ext}}_{\textnormal{e}} B(s)\;\textnormal{d}s + 4d C_d\alpha M^dt^d \leq 5dC_d\alpha M^dt^d
		\]
		when $t$ is sufficiently large. Applying Lemma~\ref{lem:rough time} with $\phi(t) = \#\partial_\textnormal{e}^{\textnormal{ext}}B(t)$ and $\psi$ equal to a constant, and taking $t\to\infty$ completes the proof.
	\end{proof}

	\subsection{Proof of Theorem~\ref{thm:holes}}
	In this section, we will show that almost surely, $\#\partial_{\textnormal{e}} B(t) \geq C[(1-F_Y(t))\vee t^{-1}]t^{d}$ for all large $t$. We will use Lemma 2.6, and we remark that although $D_4$ (from that lemma) depends on $M$ (from the statement of Theorem~\ref{thm:antalpisztora}), by a straightforward coupling, the conclusion of Theorem~\ref{thm:antalpisztora} (and hence Lemma~\ref{lem:upperbound}) still holds if we fix $D_4$ and increase $M$. Hence, we may assume $M$ is sufficiently large so that 
	\[
	\delta := D_1/16D_4M < 1/2,
	\]
	and therefore
	\[
	R := (1-\delta)^{-1} <2.
	\]
	Let $t_n = 4D_4MR^{n+1}$.
	
	We will define a set of vertices that form size-one holes in $B(t)$, and will contribute to the size of $\partial_{\textnormal{e}} B(t)$. For $r,u$ with $r < u$, let $\Ann(r,u) = S(u)\setminus S(r)$. Define $L_{n}$ to be the number of vertices $v$ in $\mathbf{e}_1 + \Ann(R^n,R^{n+1})\cap 3\Z^d$ such that (with $\mathsf{C}$ from Lemma~\ref{lem:upperbound} and open edges being those with $t_e \leq M$)
	\begin{enumerate}
		\item[(i)] $v-\mathbf{e}_1\in\mathsf{C}$, and
		\item[(ii)] all edges adjacent to $v$ have edge-weights $> t_n$.
	\end{enumerate}
	We claim that almost surely, when $n$ is sufficiently large, 
	\begin{equation}\label{eq: new_eq}
	\#\partial_{\textnormal{e}} B(t) \geq L_n \quad \text{for all } t \in [t_n,t_{n+1}).
	\end{equation}
	The reason is as follows: from Lemmas~2.4 and 2.6, we can almost surely find a random $k_0$ such that
	\begin{enumerate}
		\item whenever $\|x\|_\infty\geq k_0$ and $x\in\mathsf{C}$, $T(0,x) \leq 4D_4M\|x\|_\infty$, and
		\item whenever $\|z\|_1\geq k_0$, $T(0,z)\geq D_1 \|z\|_1$.
	\end{enumerate}
	For a given $n$, consider an outcome in the event $\{k_0\leq R^n\}$. Let $t\in [t_n, t_{n+1})$. If $x - \mathbf{e}_1\in \Ann(R^n, R^{n+1})$, and
	\begin{enumerate}
		\item[(a)] all the edges incident to $x$ have edge-weights $> (1-\delta)t$,
		\item[(b)] $T(0,x-\mathbf{e}_1)\leq t$ and
		\item[(c)] $T(0,z)\geq \delta t$ for all $z$ such that $\|z-x\|_1=1$,
	\end{enumerate}
	immediately $\{x-\mathbf{e}_1,x\} \in \partial_{\textnormal{e}} B(t)$ (see the sketch of proof of Theorem~\ref{thm:holes}). Now, condition (ii) in the definition of $L_n$ implies (a), because $t_n=4D_4MR^{n+1} = (1-\delta) 4D_4MR^{n+2} = (1-\delta)t_{n+1} >  (1-\delta)t$. Secondly, condition (i) implies (b): when $x-\mathbf{e}_1\in\mathsf{C}$,
	\[
	T(0,x-\mathbf{e}_1)\leq 4D_4M\|x-\mathbf{e}_1\|_\infty \leq 4D_4M R^{n+1} = t_n \leq t.
	\]
	(c) always holds because when $z$ is such that $\|z-x\|_1=1$, then $\|z\|_1 \geq R^n \geq k_0$ (since $\|z\|_1 > \lfloor R^n\rfloor$ and $\|z\|_1\in\N$), and hence 
	\[
	T(0,z) \geq D_1\|z\|_1\geq D_1  R^n= \delta 16D_4M R^n \geq \delta 4D_4MR^{n+2} =\delta t_{n+1} \geq \delta t.
	\]
	Therefore, the number of vertices that satisfy (a), (b) and (c) is bounded below by $L_n$, and this proves \eqref{eq: new_eq}.
	
	We will soon show that for some constant $C_5>0$, almost surely, for all $n$ large,
	\begin{equation}
	\label{eq:L_nlowerbound}
	L_n \geq 2C_5(1-F_Y(t_n))R^{nd} - C_5[(1-F_Y(t_n)) \vee R^{-n}]R^{nd}.
	\end{equation}
	Before showing \eqref{eq:L_nlowerbound} holds for all large $n$, we first show how \eqref{eq:L_nlowerbound} implies Theorem~\ref{thm:holes}. Combining \eqref{eq:L_nlowerbound} with $\#\partial_{\textnormal{e}} B(t) \geq L_{n}$, we have almost surely that for all large $n$ and for all $t\in [t_n, t_{n+1})$,
	\[
	\#\partial_{\textnormal{e}} B(t) \geq 2C_5(1-F_Y(t_n))R^{nd} - C_5[(1-F_Y(t_n)) \vee R^{-n}]R^{nd}.
	\]
	Fix such $n$ and let $t\in [t_n, t_{n+1})$. There are two cases we need to consider.
	\begin{enumerate}
		\item If $1-F_Y(t_n)\leq R^{-n}$, then $1-F_Y(t_n)\leq 16D_4M/t$, and hence $1-F_Y(t)\leq 16D_4M/t$. By \eqref{eq:volume}, $\#\partial_{\textnormal{e}} B(t) \geq c_3t^{d-1}$. In particular,
		\begin{align*}
		[(1-F_Y(t))\vee t^{-1}]t^d &\leq [16D_4M/t \vee t^{-1}]t^d \\
		&\leq C_7t^{d-1} \\
		&\leq  (C_7/c_3)\#\partial_{\textnormal{e}} B(t) 
		\end{align*}
		holds.
		
		\item $1-F_Y(t_n)> R^{-n}$. This yields $\#\partial_{\textnormal{e}} B(t) \geq C_5(1-F_Y(t_n))R^{nd} \geq C_8 (1-F_Y(t))t^d$. Again \eqref{eq:volume} gives $\#\partial_{\textnormal{e}} B(t) \geq c_3t^{d-1}$, and so combining these two inequalities we have
		\[
		\#\partial_{\textnormal{e}} B(t) \geq C_9[(1-F_Y(t))\vee t^{-1}]t^d.
		\] 
	\end{enumerate}
	Hence we have almost surely, $\#\partial_{\textnormal{e}} B(t) \geq C_{10}[(1-F_Y(t))\vee t^{-1}]t^d$ for all large $t$.
	
	It now remains to show almost surely,  \eqref{eq:L_nlowerbound} holds for all large $n$. Define $V_n$ to be the set of vertices $v$ in $\mathbf{e}_1+\Ann(R^n, R^{n+1}) \cap 3\Z^d$ such that the event $E_v$ occurs, where $E_v$ is defined by the conjunction of the following conditions:
	\begin{itemize}
		\item[(A)] $v-\mathbf{e}_1\in\mathsf{C}$, and
		\item[(B)] all the nearest-neighbor edges between vertices in $\{z\in\Z^d : \|z-v\|_\infty = 1\}$ are open.
	\end{itemize}
	By choice of $M$ and the FKG inequality \cite[Chapter~2]{Grimmett}, one has $\mathbf{P}(E_v) = \mathbf{P}(E_0)>0$ for all $v$. Letting $K_n = \#V_n$, by Birkhoff's ergodic theorem (applied to the random variables $(\mathbf{1}_{E_v})_{v \in \mathbb{Z}^d}$, there exists $C_1>0$ such that
	\[
	\frac{K_1+\cdots+K_n}{\#(S(R^{n+1})\cap 3\Z^d)} \to C_1
	\]
	as $n\to\infty$ almost surely. This implies
	\begin{equation}
	\label{eq:K_n}
	\frac{K_n}{\#(S(R^n) \cap 3\Z^d)}\to C_2 >0
	\end{equation}
	almost surely.
	
	Note that for any $\lambda_n\in\R$, 
	\begin{equation}
	\label{eq:L_n}
	\{L_{n}\leq \lambda_n\} \subseteq \{K_n \leq C_4R^{nd}\}\cup\{L_{n}\leq \lambda_n, K_n>C_4R^{nd}\}.
	\end{equation}
	When $C_4$ is small (depending on $C_2$ and $d$), then by \eqref{eq:K_n}, almost surely, for all large $n$, the first event on the right of \eqref{eq:L_n} does not occur. The probability that the second event occurs equals
	\begin{equation}\label{eq: decomposition}
	\sum_{V\subseteq \mathbf{e}_1+\Ann(R^n,R^{n+1})\cap 3\Z^d, \# V > C_4R^{dn}} \prob(L_{n}\leq \lambda_n | V_n = V) \prob(V_n = V).
	\end{equation}
	For a given finite $V\subseteq \Z^d$, let $N_{V,n}$ be the number of $v\in V$ such that all edges incident to $v$ have edge-weights $>4D_4MR^{n+1}$. Then
	\begin{equation}\label{eq: before_lemma}
	\prob(L_{n}\leq \lambda_n|V_n=V)  \leq \prob(N_{V,n}\leq \lambda_n| V_n=V).
	\end{equation}
	\begin{lem}[Shielding lemma]\label{lem: independence}
		For a given finite $V\subseteq 3\mathbb{Z}^d$, the random variable $N_{V,n}$ and the event $\{V_n=V\}$ are independent.
	\end{lem}
	\begin{proof}
		Recall that $V_n$ is the set of $v \in \Ann_n:= \mathbf{e}_1 + \text{Ann}(R^n,R^{n+1}) \cap 3\mathbb{Z}^d$ satisfying the conditions (A) and (B) above.
		Let $A_n$ and $B_n$ be the set of $v \in \Ann_n$ satisfying (A) and (B) respectively. For a given $V \subseteq \Ann_n$, let $A_n'(V)$ be the set of $v\in \Ann_n$ satisfying the condition
		\begin{enumerate}
			\item[(A')] $v-\mathbf{e}_1\to\infty$ via an open path without touching $V$.
		\end{enumerate}
		\begin{figure}[h]
			\label{fig: shield}
			\centering
			\includegraphics[width=.6\textwidth]{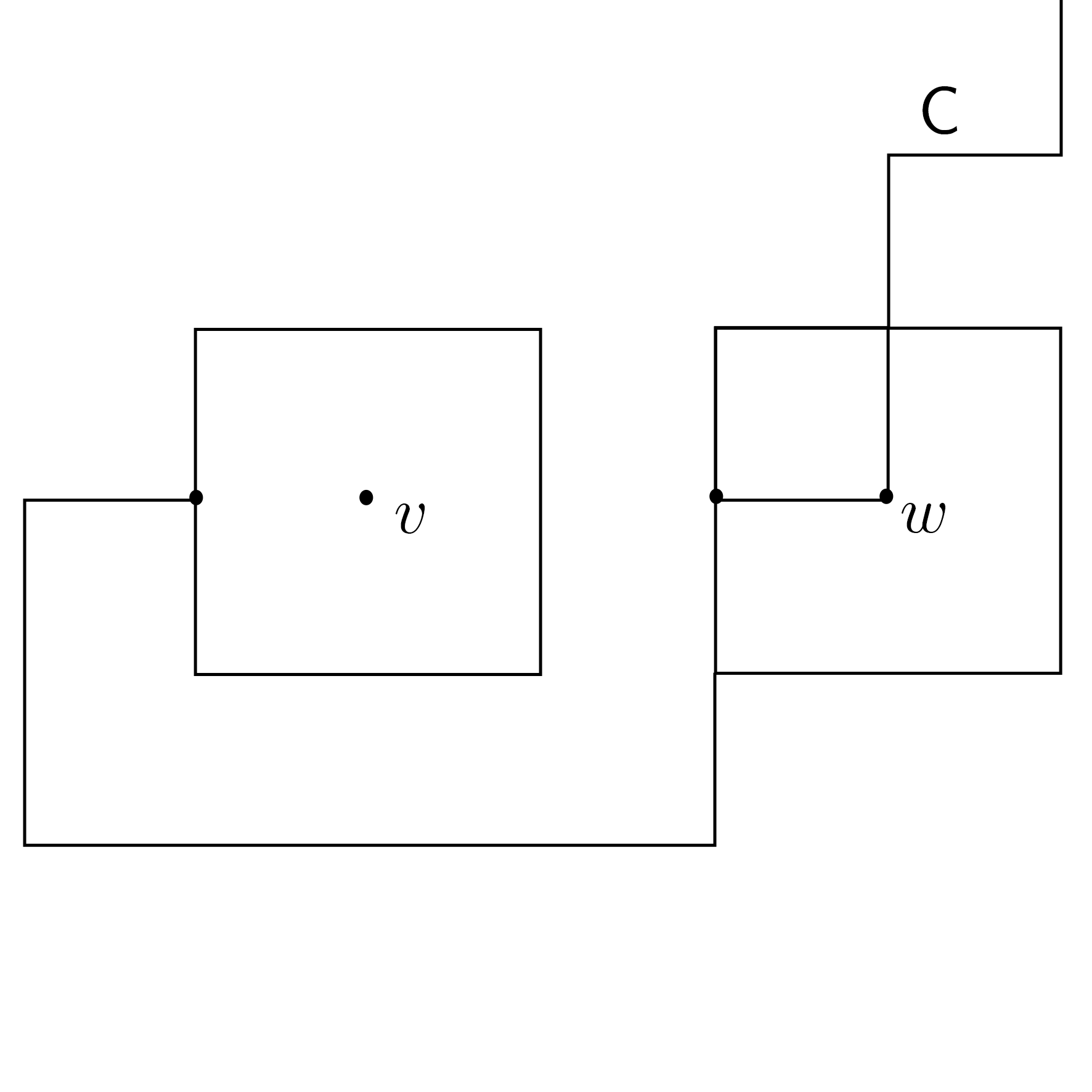}
			\caption{Depiction of the proof of Lemma~\ref{lem: independence}. All the lines represent edges in the infinite open cluster $\mathsf{C}$. $v-\mathbf{e}_1$ can be connected to $\infty$ using edges incident to $w$, but the ``shield'' surrounding $w$ can ``reroute'' the path.}
		\end{figure}
		
		We claim that for a given $V \subseteq \Ann_n$, on the event $\{V \subseteq B_n\}$, the sets $A_n$ and $A_n'(V)$ are equal. Clearly (A') implies (A), so if $V \subseteq B_n$, then $A_n'(V) \subseteq A_n$. On the other hand, if $V \subseteq B_n$ and $v \in A_n$, then because the edges in (B) form ``shields'' around all $w\in V$, any infinite open path starting from $v-\mathbf{e}_1$ and taking an edge incident to a $w \in V$ may be ``rerouted'' around $w$, using edges described in (B) instead of those incident to $w$. (See Figure~\ref{fig: shield}.) Here we are using the fact that $v-\mathbf{e}_1$ is not in $V$ (as $v$ and $V$ are in the lattice $3\mathbb{Z}^d$) and so any such path does not begin at a vertex of $V$. Therefore in this setting, any $v$ satisfying (A) also satisfies (A') and this shows the claim.
		
		Now, the random variable $N_{V,n}$ and the conditions (A') and (B) depend on two disjoint sets of edges, and hence they are independent: for any $r \in \mathbb{R}$ and $V \subseteq \Ann_n$,
		\begin{align*}
		\mathbf{P}(N_{V,n} = r, V_n = V) & = \mathbf{P}(N_{V,n} = r, A_n \cap B_n =V) \\
		&= \mathbf{P}(N_{V,n}=r, A_n'(V) \cap B_n = V) \\
		&= \mathbf{P}(N_{V,n}=r)\mathbf{P}(A_n'(V) \cap B_n = V) \\
		&= \mathbf{P}(N_{V,n}=r) \mathbf{P}(V_n=V).
		\end{align*}
	\end{proof}
	Returning to \eqref{eq: before_lemma}, by Lemma~\ref{lem: independence},
	\begin{equation}\label{eq: conditional}
	\prob(N_{V,n} \leq \lambda_n| V_n=V) = \prob(N_{V,n} \leq \lambda_n).
	\end{equation}
	Note that $N_{V,n}$ is just a sum of i.i.d. Bernoulli random variables (say $N_{V,n} = X_1+\cdots+X_k$) with parameter $1-F_Y(t_n)$, and hence there exist $C_5, C_6>0$ (one can take $C_5 = C_4/2$) such that if $\#V > C_4 R^{dn}$, then
	\begin{enumerate}
		\item $\E N_{V,n} = k(1-F_Y(t_n)) \geq 2C_5 (1-F_Y(t_n))R^{nd}$,
		\item $\sum_{i=1}^k \E X_i^2 = \sum_{i=1}^k \E X_i \leq C_6 (1-F_Y(t_n))R^{nd}$.
	\end{enumerate}
	Thus by Bernstein's inequality,
	\begin{align}
	&\phantom{\leq\;\;}\prob(N_{V,n} \leq \E N_{V,n} - C_5[(1-F_Y(t_n))\vee R^{-n}]R^{nd}) \nonumber\\&\leq 2\exp\left(- \frac{C_5^2[(1-F_Y(t_n))\vee R^{-n}]R^{nd}}{2(C_5/3+C_6)}\right) \nonumber\\
	&\leq 2\exp\left(- \frac{C_5^2R^{(d-1)n}}{2(C_5/3+C_6)}\right). \label{eq: pizza}
	\end{align}
	
	Combining \eqref{eq: decomposition}, \eqref{eq: before_lemma}, \eqref{eq: conditional}, and \eqref{eq: pizza}, and using item 2 above, we see that if we put
	\[
	\lambda_n = 2C_5(1-F_Y(t_n))R^{nd} - C_5[(1-F_Y(t_n)) \vee R^{-n}] R^{nd},
	\]
	then
	\begin{align*}
	\sum_n \mathbf{P}(L_n \leq \lambda_n, K_n > C_4R^{nd}) &\leq \sum_n \sum_{\#V > C_4R^{dn}} \mathbf{P}(N_{V,n} \leq \lambda_n) \mathbf{P}(V_n = V) \\
	&\leq \sum_n \left[2 \exp\left( - \frac{C_5^2R^{(d-1)n}}{2(C_5/3+C_6)} \right) \sum_V \mathbf{P}(V_n=V)\right] < \infty.
	\end{align*}
	
	Hence, by the remarks below \eqref{eq:L_n} and the Borel-Cantelli lemma, almost surely,
	\[
	L_{n}\geq 2C_5(1-F_Y(t_n))R^{nd} - C_5[(1-F_Y(t_n))\vee R^{-n}]R^{nd}
	\]
	for all large $n$. This proves the desired inequality \eqref{eq:L_nlowerbound} and hence Theorem~\ref{thm:holes}.
	
	\subsection{Proof of Theorem~\ref{thm:curv}}
	In this section, we will assume that $\prob(t_e=0)<p_c$, $\E e^{\alpha t_e}<\infty$ for some $\alpha>0$ and that $\mathcal{B}$ satisfies the uniform curvature condition. We will need to control geodesics, so we first show the following lower bound on the Busemann-type function $T(0,kx) - T(0,\ell x)$:
	\begin{prop}
		\label{prop: unif_curv}
		There exist $C_1, C_2>0$ and $C_3\in (0,1)$ such that for any $x\in\R^d$ with $\|x\|_2=1$ and for any $k, \ell\geq 0$ with $k\geq \ell$,
		\[
		\prob(T(0,kx)-T(0,\ell x)\geq C_1(k-\ell))\geq 1-C_2e^{-(k-\ell)^{C_3}}.
		\]
	\end{prop}
	
	Let us begin with some definitions introduced in \cite{newman}:
	\begin{df}
		With $\eta$ from the curvature assumption, Definition~\ref{def: curvature}, let $\delta \in \left( 0, \frac{1}{2\eta}\right)$.
		\begin{enumerate}
			\item For $x,y\in\R^d\setminus\{0\}$, let $\theta(x,y)$ be the angle (in $[0,\pi]$) between $x$ and $y$.
			\item For a vertex $y\neq 0$, define
			\[
			C_y = \{x\in\Z^d: g(x)\in[g(y) - g(y)^{1-\eta\delta}, 2g(y)], \theta(x,y)\leq g(y)^{-\delta}\}.
			\]
			\item $\mathrm{out}(x)$ is the set of vertices $z$ such that $T(0,z) = T(0,x)+T(x,z)$, or equivalently, the set of vertices in some geodesic from $0$ that goes through $x$.
			\item Define $\partial_{\mathrm{i}}C_y$ (resp. $\partial_{\mathrm{o}}C_y$) to be the set of boundary vertices in $C_y$ with $g(x)<g(y) - g(y)^{1-\eta\delta}$ (resp. $>2g(y)$). Also define $\partial_{\mathrm{s}}C_y$ to be the set of boundary vertices in $C_y$ with $\theta(x,y)>g(y)^{-\delta}$.
			\item Define $G_y = \{\mathrm{out}(y)\cap(\partial_{\mathrm{i}}C_y \cup \partial_{\mathrm{s}}C_y)\neq \emptyset\}$.
		\end{enumerate} 
	\end{df}
	
	The events $G_y$ help to control geodesic wandering (see \eqref{eq: lasagna} below). 
	\begin{lem}
		There exist constants $C_4, C_5>0$ such that 
		\[
		\prob(G_y) \leq C_4\exp(-C_5\|y\|_2^{\frac{1}{2}-\eta\delta}).
		\]
	\end{lem}
	\begin{proof}
		The proof is identical to that of \cite[Proposition~3.2]{newman} with $2\delta$ replaced by $\eta \delta$.
	\end{proof}
	
	\begin{proof}[Proof of Proposition~\ref{prop: unif_curv}]
		It will suffice to show the result for $k-\ell$ sufficiently large (independently of $x$). Let $x\in \R^d$ with $\|x\|_2=1$ and let $k, \ell\geq 0$ with $k> \ell$. Let $\gamma = ([kx]=x_0,e_1,x_1,\ldots,e_r,x_r=0)$ be a geodesic from $[kx]$ (the point of $\mathbb{Z}^d$ with $kx \in [kx]+[0,1)^d$) to $0$. For $z\in\Z^d$, define $T_z$ to be translation operator by $z$; that is, $T_z((t_e)) = (t_{e-z})$. Define $G_z' = T_{x_0}G_{z-x_0}$ to be the shifted $G$ event. Then we have
		\[
		\prob(G_z') \leq C_4e^{-C_5\|x_0-z\|_2^{\frac{1}{2}-\eta\delta}}
		\] 
		and if we define
		\[
		A_M = \{G_z'^c \text{ occurs for all $z$ with $\|z-x_0\|_2\geq M$}\},
		\]
		then
		\begin{equation}\label{eq: taco}
		\prob(A_M) \geq 1-C_6e^{-C_7M^{\frac{1}{2}-\eta\delta}}.
		\end{equation}
		
		Let $H$ be the hyperplane which is perpendicular to $x_0$ and passes through $\ell x$. Let $y$ be the first vertex in $\gamma$ contained in $H$ or in the component of $H^c$ containing $0$. Now set $M=(k-\ell)/2$ so that 
		\begin{equation}\label{eq: taco_3}
		\|y-x_0\|_2\geq M
		\end{equation}
		(if $k-\ell$ is large). Using this and the proof of \cite[Proposition~3.2]{newman}, one can show that on $A_M$, for some $C_8$ independent of $x,k,\ell$, one has
		\begin{equation}\label{eq: lasagna}
		|\theta(y-x_0, x_r-x_0)| \leq C_8\|y-x_0\|_2^{-\delta}.
		\end{equation}
		\begin{figure}[h]
			\centering
			\includegraphics[width=.6\textwidth]{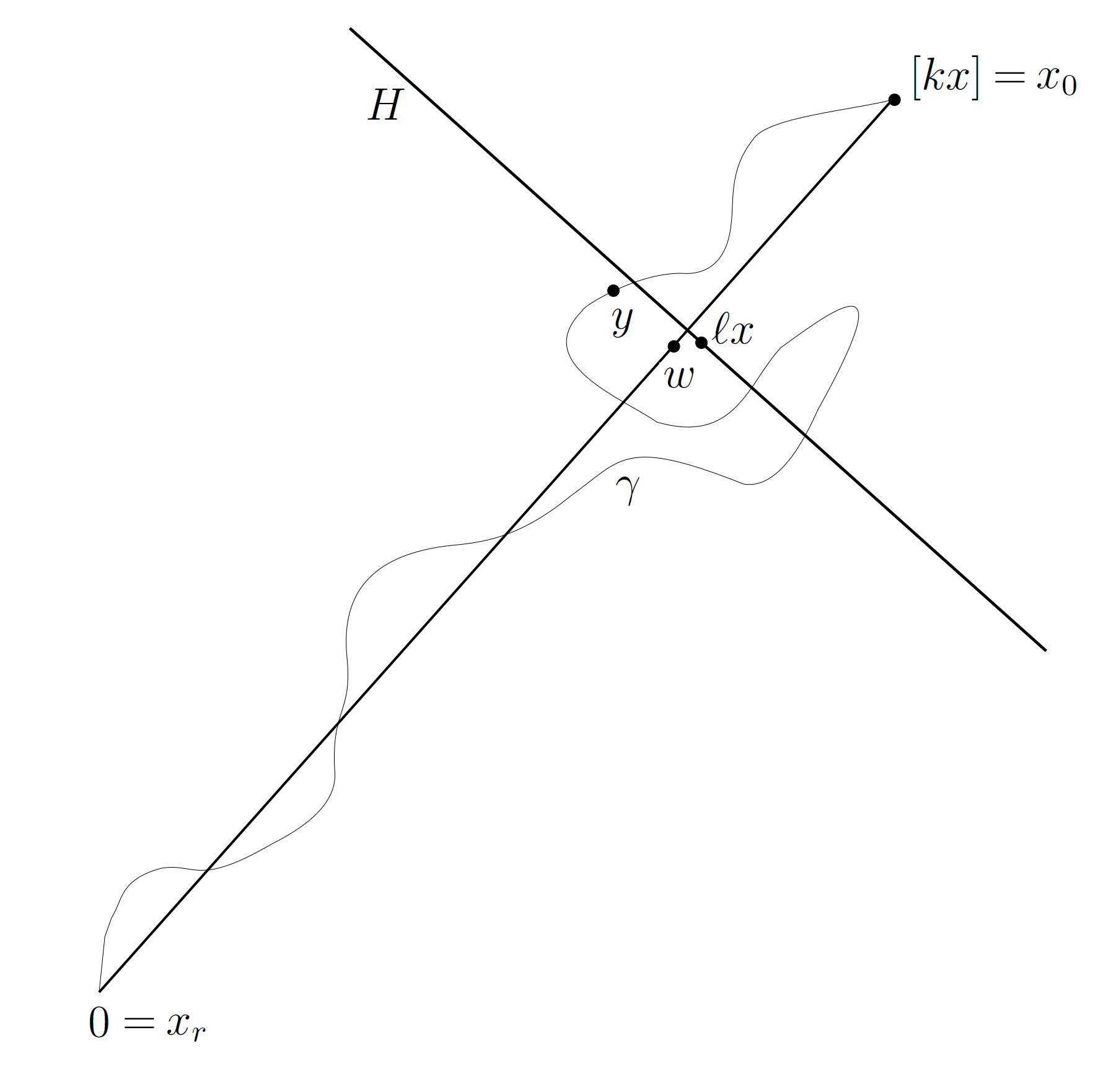}
			\caption{Depiction of the proof of Proposition~\ref{prop: unif_curv}. $\gamma$ is a geodesic from $[kx]$ to $0$ and $y$ is the first vertex of $\gamma$ after $\gamma$ passes through $H$. Because of the curvature assumption, $y$ is close to $\ell x$.}
			\label{fig: fig_2}
		\end{figure}
		Let $w$ be the orthogonal projection of $y$ to the line spanned by $x_0$. Then clearly we have $\|w-\ell x\|_2, \|x_0-kx\|_2\leq C(d)$ for some $C(d)$ depending only on $d$. So
		\begin{align}
		\|y-\ell x\|_2 &\leq \|y-w\|_2 + \|w - \ell x\|_2 \nonumber \\
		&\leq \|x_0-w\|_2\tan{|\theta(y-x_0, x_r-x_0)|} + C(d) \nonumber \\
		&\leq (\|x_0-kx\|_2+\|kx-\ell x\|_2+\|\ell x-w\|_2)\tan{|\theta(y-x_0, x_r-x_0)|} + C(d)\nonumber \\
		&\leq C_9(k-\ell)^{1-\delta}\nonumber\\
		&= C_{10}M^{1-\delta} \label{eq: taco_2}
		\end{align}
		if $k-\ell$ is sufficiently large. Let $B_{M,D}$ be the event that
		\begin{enumerate}
			\item for any $u$ with $\|u-x_0\|_2\geq M$, $T(u,x_0) \geq D\|u-x_0\|_2$,
			\item for any $u$ with $\|u-\ell x\|_2\leq C_{10}M^{1-\delta}$, $T(\ell x, u) \leq\frac{D}{2}M$.
		\end{enumerate}
		\begin{lem}\label{lem: lem_head}
			There exist $C_{11}, C_{12}, C_{13}>0$ such that $\prob(B_{M, C_{11}}^c)\leq C_{12}e^{-C_{13}M}$. 
		\end{lem}
		\begin{proof}
			By Lemma~\ref{lem:largedeviation} and the fact that all norms on $\R^d$ are equivalent, there exist constants $C_{14}, C_{15}>0$ such that for all $z\in \Z^d$,
			\[
			\prob(T(0,z) \leq C_{11}\|z\|_2)\leq C_{14}e^{-C_{15}\|z\|_2}.
			\]
			
			This implies for any $u$ with $\|u-x_0\|_2 \geq M$,
			\begin{equation}\label{eq: tacos_1}
			\prob(T(u,x_0)\leq C_{11}\|u-x_0\|_2) \leq C_{14}e^{-C_{15}\|u-x_0\|_2} \leq C_{14}e^{-C_{15}M}.
			\end{equation}
			
			On the other hand, let $u$ be such that $\|u-\ell x\|_2\leq C_{10} M^{1-\delta}$. Recall that $\E e^{\alpha t_e}<\infty$. By bounding $T(u,\ell x)$ above by the passage time of a deterministic path with $\|u-[\ell x]\|_1$ many edges, we have
			\begin{align*}
			\prob\left(T(u,\ell x)\geq \frac{C_{11}}{2}M\right) &\leq \exp\left(-\frac{C_{11}}{2}\alpha M\right)\left(\E e^{\alpha t_e}\right)^{\|u-[\ell x]\|_1} \\
			&\leq \exp\left(-\frac{C_{11}}{2}\alpha M\right)\left(\E e^{\alpha t_e}\right)^{C_{16}M^{1-\delta}} \\
			&= \exp\left(-\frac{C_{11}}{2}\alpha M + (C_{16}M^{1-\delta})\log{\E e^{\alpha t_e}}\right).
			\end{align*}
			So for $M$ sufficiently large, we have
			\[
			\prob\left(T(u,\ell x)\geq \frac{C_{11}}{2}M\right) \leq e^{-C_{17} M},
			\]
			and combining this with \eqref{eq: tacos_1}, we obtain
			\[
			\prob(B_{M, C_{11}}^c)\leq C_{12}e^{-C_{13}M}. 
			\]
		\end{proof}
		
		On $A_M\cap B_{M, C_{11}}$, we use \eqref{eq: taco_3} and \eqref{eq: taco_2} to estimate
		\begin{align*}
		T(0,kx) - T(0,\ell x) &= T(0,y) + T(y,kx) - T(0,\ell x)\\
		&\geq T(y,kx) - T(y,\ell x)\\
		&=T(y,x_0)-T(y,\ell x)\\
		&\geq C_{11}\|y-x_0\|_2 - \frac{C_{11}}{2}M\\
		&\geq \frac{C_{11}}{4}(k-\ell).
		\end{align*}
		From this, \eqref{eq: taco}, and Lemma~\ref{lem: lem_head}, we conclude
		\[
		\prob(T(0,kx) - T(0,\ell x)\geq C_{1}(k-\ell)) \geq 1-C_{2}e^{-(k-\ell)^{C_{3}}},
		\]
		where $C_3 = \frac{1}{2} - \eta\delta \in (0,1)$.
	\end{proof}
	
	Fix a large $\lambda>0$ and let $F_t$ be the event that for any edge $e \subseteq 2t\mathcal{B}$, $t_e\leq \lambda\log{t}$. Further define, for $n\in\mathbb{N}$, $\tilde{F}_n$ to be the event that for any edge $e \subseteq (2n+2)\mathcal{B}$, $t_e\leq \lambda\log{n}$. Note that for $N\in\mathbb{N}$ large, if $F_t$ does not occur for some $t\geq N$, then $\tilde{F}_n$ does not occur for some $n\geq N$ (namely $n=\lfloor t\rfloor$). Therefore for $\lambda > \frac{1}{\alpha}(d+3)$,
	\begin{align*}
	\prob(F_t \text{ does not occur for some $t\geq N$}) &\leq  \sum_{n \geq N} \sum_{e \subseteq (2n+2)\mathcal{B}} \mathbf{P}(t_e > \lambda \log n ) \\
	&\leq C_{17}\sum_{n\geq N} n^{d}\exp\left(-\alpha\lambda\log{n}\right) \\
	&\leq \frac{C_{18}}{N^{-d-1+\lambda\alpha}} \\
	&\leq \frac{C_{18}}{N^{2}}.
	\end{align*}
	By the Borel-Cantelli lemma, almost surely, $F_t$ occurs for all large $t$.
	
	For $t>1$ and $c>0$ to be determined, define
	\[
	\Ann'(t) = (t+c{t}^{1/2}\log{t})\mathcal{B} \setminus (t-c{t}^{1/2}\log{t})\mathcal{B}.
	\]
	We will decompose $\Ann'(t)$ using rays, and count the intersection of $\partial_\textnormal{e} B(t)$ with these rays.
	
	\begin{lem}
		\label{lem:vectors}
		There exists $C_{19}>0$ such that for each $s>1$, there is a choice of at most $C_{19}s^{d-1}$ unit vectors $v$ such that each cube $y+[0,1)^d$, $y \in \mathbb{Z}^d$, that is completely contained in $s\mathcal{B}$ is intersected by at least one of the rays $S_{v} = \{tv: t\geq 0\}$.
	\end{lem}
	\begin{proof}
		Choose any collection of at most $C_{19}s^{d-1}$ points $y_1,\ldots, y_r\in \partial s\mathcal{B}$ such that for any $y\in \partial s\mathcal{B}$, there exists $j\in \{1,\ldots,r\}$ such that $\|y-y_j\|_2\leq  \frac{1}{2}$. Define $v_i = y_i/\|y_i\|_2$, $i=1,\ldots,r$ and, letting $x_1, \ldots, x_k$ be the midpoints of the cubes in $s \mathcal{B}$ ($k$ depends on $s$). Define $z_j = sx_j/g(x_j)$, so that $z_j\in \partial s\mathcal{B}$. Choose $m$ such that $\|z_j - y_m\|_2\leq  \frac{1}{2}$ and note that $y_m\in S_{v_m}$, so the distance between $z_j$ and $S_{v_m}$ is at most $\frac{1}{2}$. Using similar triangles one can see that the distance between $x_j$ and $S_{v_m}$ is also at most $\frac{1}{2}$, which proves the lemma. 
	\end{proof}
	
	For $s = 2t$, let $v_1,\ldots, v_r$ be the corresponding unit vectors from Lemma~\ref{lem:vectors} and define $S_i(t) = S_{v_i}\cap \Ann'(t)$. If $t$ is large then for each $y\in \Z^d\cap \Ann'(t)$, there exists $i$ such that $y+[0,1)^d$ intersects $S_i(t)$. We define
	\[
	\tilde S_i(t) = \{y\in \Z^d\cap \Ann'(t): y+[0,1)^d\text{ intersects $S_i(t)$}\}.
	\] 
	
	For each $y\in \Z^d\cap \Ann'(t)$, let $R_y(t,\rho)$ be the set of points $v\in \tilde S_i(t)$ satisfying $\|v-y\|_2\leq\rho$. 
	\begin{figure}[h]
		\centering
		\includegraphics[width=.75\textwidth]{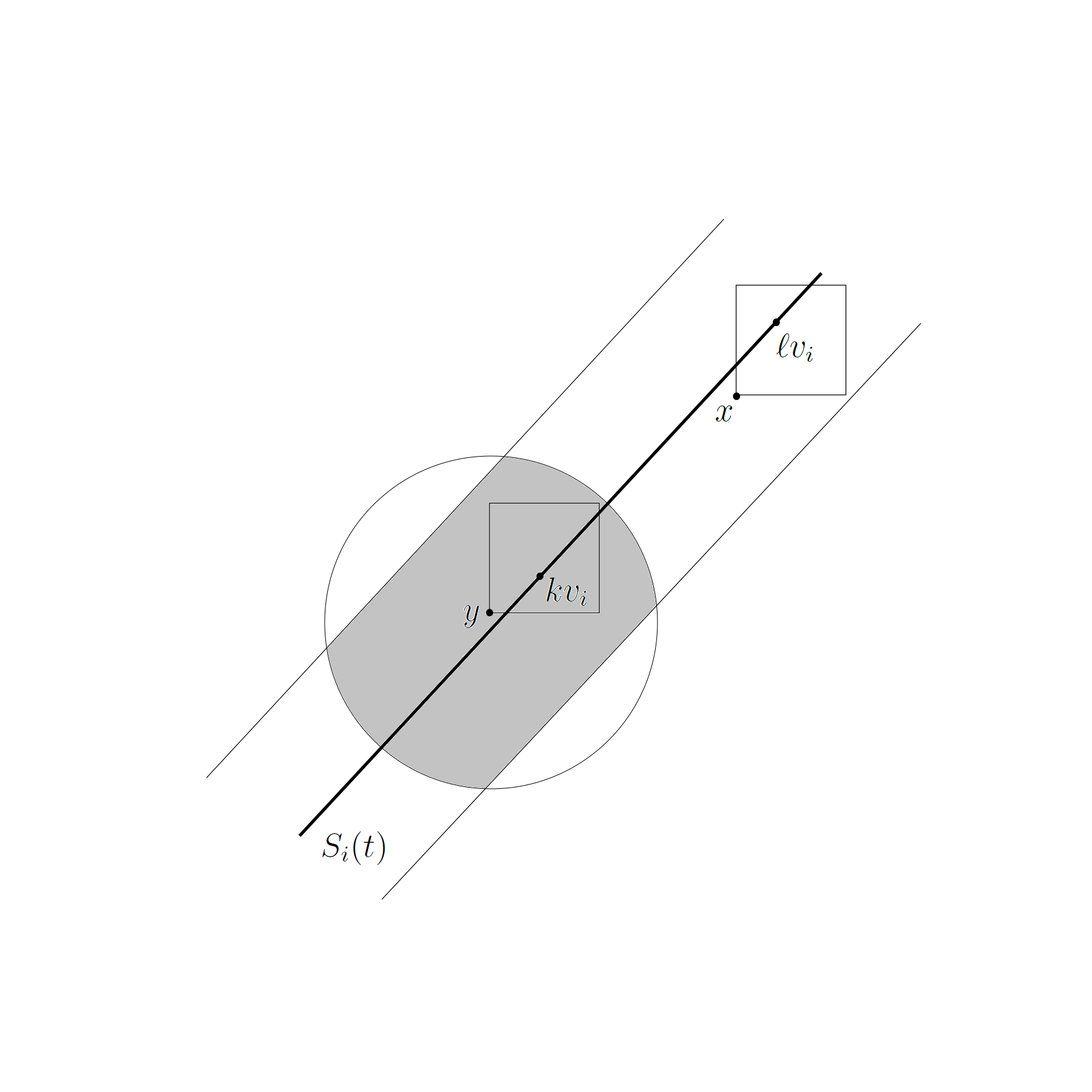}
		\caption{The shaded region is $R_y(t,2\rho)$, and the region between the two light-colored lines represents the lattice points that comprise $\tilde S_i(t)$. The squares depict $x+[0,1)^d$ and $y+[0,1)^d$ for $x,y \in \tilde S_i(t)$. Note that when $x$ and $y$ are far away from each other, then so are the corresponding points $kv_i$ and $\ell v_i$ on the ray $S_i(t)$. Also, $|T(0,x)-T(0,y)| = |T(0,kv_i)-T(0,\ell v_i)|$.}
	\end{figure}
	For $C_1$ from Proposition~\ref{prop: unif_curv}, define $G_y(t,\rho)$ to be the event that for any $x\in \tilde S_i(t)\setminus R_y(t,\rho)$, we have $|T(0,y)-T(0,x)|\geq C_1\rho/2$. Note that when $\rho$ is large enough (depending on the dimension), the inequality $\|kv_i-\ell v_i\|_2 >\rho$ is implied by $\|y_k-y_\ell\|_2>2\rho$, where $y_k, y_\ell\in \tilde S_i(t)$ are such that $kv_i\in y_k+[0,1)^d$ and $\ell v_i\in y_\ell + [0,1)^d$. Therefore, when $\rho$ is large, $G_y(t, 2\rho)$ contains the event that for any $k, \ell$ with $\|kv_i-\ell v_i\|_2 >\rho$ and $kv_i \in y+[0,1)^d$, one has $|T(0,kv_i)-T(0,\ell v_i)| \geq C_1\rho$. As there are at most $O(t^{1/2}\log{t})$ many points in $\tilde S_i(t)$, by Proposition~\ref{prop: unif_curv}, for $t$ sufficiently large,
	\[
	\prob(G_y(t,(\lambda\log{t})^{1/C_3})) \geq 1 - C_{20}t^{1/2}\log{t}\cdot e^{-\lambda\log{t}} = 1-\frac{C_{20}\log{t}}{t^{\lambda-1/2}}.
	\]
	This means
	\[
	\prob\left( \bigcap_{y \in \Ann'(t) \cap \mathbb{Z}^d} G_y(t,(\lambda \log t)^{1/C_3})\right) \geq 1- C_{21}\frac{\log t}{t^{\lambda - 1/2 - d}}.
	\]
	If $\lambda$ is chosen large enough, another discretizing argument (similar in spirit to the one applied to $F_t$ above) can show that $\bigcap_{y\in \Ann'(t)\cap \Z^d}G_y(t,(\lambda\log{t})^{1/C_3})$ occurs for all large $t$ almost surely.
	
	We moreover define
	\[
	E_t = \{(t-ct^{1/2}\log{t})\mathcal{B} \subseteq \bar{B}(t) \subseteq (t+ct^{1/2}\log{t})\mathcal{B}\}.
	\]
	By \eqref{eq:alexander}, for some $c>0$, $E_t$ occurs for all large $t$ almost surely.
	
	Now suppose that $E_t$, $F_t$ and $G_y(t,(\lambda\log{t})^{1/C_3})$ occur for all $y\in\Ann'(t)\cap \Z^d$ and write
	\[
	\#\partial_{\textnormal{e}} B(t)  \leq \sum_{i=1}^r \# [\partial_{\textnormal{e}} B(t)\cap E(\tilde S_i(t))],
	\]
	where $E(\tilde S_i(t))$ is the set of edges with at least one endpoint in $\tilde S_i(t)$. If $\#[\partial_{\textnormal{e}} B(t)\cap E(\tilde S_i(t) )]> 0$, choose the first $e \in \partial_{\textnormal{e}} B(t)\cap E(\tilde S_i(t))$ in some deterministic ordering. Write $e=\{x,y\}$ and assume, without loss of generality, that  $x \in B(t)$. Then we have for $t$ large,
	\[
	t < T(0,y) \leq T(0,x) + t_e \leq t + \lambda\log{t} < t + \frac{C_1}{8} (\lambda\log{t})^{1/C_3}.
	\]
	The third inequality holds because $T(0,x)\leq t$ and we have assumed that $F_t$ occurs. The last inequality uses $C_3 < 1$. If $f=\{w,z\}$ is another edge in $\partial_{\textnormal{e}} B(t)\cap E(\tilde S_i(t))$, we also have the inequality
	\[
	t< T(0,z) < t+\frac{C_1}{8} (\lambda\log{t})^{1/C_3}
	\]
	for $t$ large if $w \in B(t)$. In such a case, we have $|T(0,y) - T(0,z)| < \frac{C_1}{4} (\lambda\log{t})^{1/C_3}$, and furthermore
	\[
	\max_{a \in \{x,y\}, b \in \{w,z\}} |T(0,a) - T(0,b)| < \frac{C_1}{4} (\lambda \log t)^{1/C_3} + t_e + t_f < \frac{C_1}{2} (\lambda \log t)^{1/C_3}.
	\]
	Supposing without loss of generality that $z,y \in \tilde S_i(t)$, this (along with the occurrence of $G_y(t,(\lambda \log t)^{1/C_3})$) implies $z\in R_y(t,(\lambda\log{t})^{1/C_3})$. By construction of $S_i(t)$, there are at most $C_{22}(\lambda\log{t})^{1/C_3}$ many vertices in $R_y(t,(\lambda\log{t})^{1/C_3})$, so
	\[
	\#[\partial_{\textnormal{e}} B(t)\cap E(\tilde S_i(t))] \leq C_{23} (\lambda\log{t})^{1/C_3},
	\]
	which means
	\[
	\#\partial_{\textnormal{e}} B(t)\leq \sum_{i=1}^r \#[\partial_{\textnormal{e}} B(t)\cap E(\tilde S_i(t))] \leq C_{23}(\lambda\log{t})^{1/C_3}r \leq C_{24} (\lambda\log{t})^{1/C_3}t^{d-1}.
	\]
	This inequality holds almost surely for all large $t$, which proves the theorem.
	
	\section{Open questions}
	In Theorem~\ref{thm:main}, we show that under the condition $\prob(t_e=0)<p_c$, almost surely, $\#\partial_\textnormal{e} B(t) \leq Ct^{d-1}\E[Y\wedge t]$ and $\#\partial_\textnormal{e}^{\textnormal{ext}} B(t) \leq Ct^{d-1}$ for a large fraction of time.
	\begin{question}
		 Is it true that almost surely, $\#\partial_\textnormal{e} B(t) \leq Ct^{d-1}\E[Y\wedge t]$ and $\#\partial_\textnormal{e}^{\textnormal{ext}} B(t) \leq Ct^{d-1}$ for all large $t$? 
	\end{question}

	Combining Theorems~\ref{thm:main} and \ref{thm:holes}, we have bounds of the form
	\[
	\left[ t(1-F_Y(t)) \vee 1\right] t^{d-1} \lesssim \#\partial_{\textnormal{e}} B(t) \lesssim \mathbf{E}[Y \wedge t] t^{d-1}
	\]
	and we have seen (near equation~\eqref{eq: nachos_bellegrande}) that although the upper and lower bounds are of the same order for most distributions, they can be quite different for distributions with highly irregular tails.
	\begin{question}
		For heavy-tailed and irregular distributions, what is the correct order of $\#\partial_\textnormal{e}B(t)$?
	\end{question}

	Last, under the uniform curvature assumption and an exponential moment condition, we obtain almost surely,
	\[ 
	\#\partial_\textnormal{e} B(t) \leq C(\log{t})^Ct^{d-1} \quad \text{for all large }t.
	\]
	\begin{question}
		Can one remove the $\log$ term under further or possibly stronger assumptions?
	\end{question}


\end{document}